\documentclass[10pt]{amsart}
\usepackage{amssymb,amsmath,amsfonts,amscd}

\def\eps{\epsilon}

\usepackage{pdfsync}
\usepackage{enumerate}

\usepackage{color}

\theoremstyle{plain}
\newtheorem{theorem}{Theorem}[section]
\newtheorem{proposition}[theorem]{Proposition}
\newtheorem{corollary}[theorem]{Corollary}

\newtheorem{lemma}[theorem]{Lemma}
\theoremstyle{definition}
\newtheorem{definition}[theorem]{Definition}

\theoremstyle{remark}
\newtheorem{remark}[theorem]{Remark}

\newcommand{\rn}[1]{{\mathbb R}^{#1}}
\newcommand{\R}{\mathbb R}

\newcommand{\G}{\mathbb G}

\newcommand{\supp}{\mathrm{supp}\;}

\newcommand{\dom}{\mathrm{Dom}\;}
\newcommand{\Ast}{\un{\ast}}

\newcommand{\he}[1]{{\mathbb H}^{#1}}

\newcommand{\cov}[1]{{\bigwedge\nolimits^{#1}{\mfrak h}}}

\newcommand{\covw}[2]{{\bigwedge\nolimits^{#1,#2}{\mfrak h}}}

\newcommand{\covH}[1]{{\bigwedge\nolimits^{#1}{\mfrak h}}}
\newcommand{\vetH}[1]{{\bigwedge\nolimits_{#1}{\mfrak h}}}

\newcommand{\covh}[1]{{\bigwedge\nolimits^{#1}{\mfrak h_1}}}

\newcommand{\scal}[2]{\langle {#1} , {#2}\rangle}

\newcommand{\Scal}[2]{\langle {#1} \vert {#2}\rangle}
\newcommand{\scalp}[3]{\langle {#1} , {#2}\rangle_{#3}}

\newcommand{\res}{\mathop{\hbox{\vrule height 7pt width .5pt depth 0pt
\vrule height .5pt width 6pt depth 0pt}}\nolimits}

\newcommand{\ccheck}{{\vphantom i}^{\mathrm v}\!\,}

\newcommand{\mc}{\mathcal }

\newcommand{\mfrak}{\mathfrak}

\newcommand{\un}[1]{\underline{#1}}

\usepackage{ifpdf}
\ifpdf
\usepackage[backref]{hyperref}
\else
\usepackage[hypertex]{hyperref}
\fi
\usepackage{pdfsync,verbatim}
\usepackage{xcolor}

\newcommand{\average}{{\mathchoice {\kern1ex\vcenter{\hrule height.4pt
width 6pt depth0pt} \kern-9.7pt} {\kern1ex\vcenter{\hrule
height.4pt width 4.3pt depth0pt} \kern-7pt} {} {} }}

\title[]{On the heat kernel of the Rumin complex and Calder\'on reproducing formula}
\author{Paolo Ciatti}
\address{Dipartimento di Matematica}
\email{paolo.ciatti@unipd.it}
\author{Bruno Franchi}
\address{Dipartimento di Matematica, Piazza di Porta S. Donato 5, 401
26 Bologna, Italy}
\email{bruno.franchi@unibo.it}
\author{Yannick Sire}
\address{Johns Hopkins University, Krieger Hall, 3400 N. Charles St., Baltimore, MD, 21218, USA}
\email{ysire1@jhu.edu}

\begin{document}

\thanks{The third author is partially supported by the Simons foundation and the NSF grant DMS-2154219. This collaboration started at the occasion of a visit of the second author at Johns Hopkins University. He would like to thank the department of mathematics for the hospitality. 
 }



\begin{abstract}
We derive several properties of the heat equation with the Hodge operator associated with the Rumin complex on Heisenberg groups and prove several properties of the fundamental solution. As an application, we use the heat kernel for Rumin's differential forms to construct a Calder\'on reproducing formula on Rumin forms. 
\end{abstract}

\maketitle

\tableofcontents

\section{Introduction and basic objects}\label{introduction}

Heisenberg groups $\he n$, $n\geq 1$, are  connected, simply connected Lie groups whose Lie algebra is the central extensions 
\begin{equation}\label{strat intro}
\mathfrak{h}=\mathfrak{h}_1\oplus\mathfrak{h}_2,\quad\mbox{with $\mathfrak{h}_2=\R=Z(\mathfrak{h})$,}
\end{equation}
with bracket $\mathfrak{h}_1\otimes\mathfrak{h}_1\to\mathfrak{h}_2=\R$ being a non-degenerate skew-symmetric 2-form. 
 Due to its 
stratification \eqref{strat intro}, the Heisenberg Lie algebra admits a one parameter group of automorphisms $\delta_t$,
\begin{eqnarray*}
\delta_t=t\textrm{ on }\mathfrak{h}_1,\quad \delta_t=t^2 \textrm{ on }\mathfrak{h}_2,
\end{eqnarray*}
which are counterparts of the usual Euclidean dilations in $\rn N$.  The stratification
of the Lie algebra $\mfrak h$ yields a lack of homogeneity of de Rham's exterior differential
with respect to group dilations $\delta_\lambda$. The so-called Rumin's complex is meant
precisely to bypass the lack of homogeneity of de Rham's complex through a new complex
that is still homotopic to de Rham's complex. In Appendix A we shall provide a more
exhaustive description of Rumin's complex.

In this paper, we investigate several properties of the heat kernel associated with the Rumin's complex on Heisenberg groups and derive a natural reproducing  formula in the spirit of Calder\'on reproducing formula. Besides this application, which has its own interest, we collect several basic results on the heat kernel which seem to be not { all} available in the literature
{(in this spirit, we quote  \cite{albinQuan} for the heat kernel in contact manifolds)} .

 This project grew out of understanding compensation-compactness phenomena for differential forms on nilpotent groups. {{ Several div-curl lemmata
  have been proved in the setting of Heisenberg groups by  the second author 
 jointly with different 
 co-authors in e.g. \cite{FTT}, \cite{BFTT}, \cite{BFT4} or \cite{BFTIndiana}. {  The present paper stems from the following observation:} in a very interesting paper, Lou and McIntosh  \cite{louMcintosh} introduced Hardy spaces of {\sl exact} differential forms for the De-Rham complex on  Euclidean spaces and generalized the foundational work of Coifman, Liosn, Meyer and Semmes \cite{CLMS}. The work of Lou and McIntosh contains several ideas around the use of differential forms with coefficients in a suitable Hardy spaces (and their atomic decompositons) but also their analysis via a reproducing formula {\sl \`a la Calder\'on} }}. Thanks to the works of the second author with Baldi and Pansu \cite{BFP1,BFP2,BFP3,BFP4}, several important functional  inequalities are now available for the Rumin complex. However, as mentioned,  the full generalization to the Rumin's complex of div-curl lemma of Lou and McInstosh requires the introduction of Hardy spaces and their atomic decomposition. At this point the theory of such spaces for the Rumin's complex departs from the Euclidean setting, even if every Heisenberg group is an homogeneous space, because of the structural difficulties inherent to the Rumin's complex. {In a subsequent paper, we will address the construction of such spaces and the applications to compensated compactness on the Rumin complex. This application to div-curl lemmas is also the motivation behind we chose to present the Calder\'on reproducing formula in the space $L^1$. However, we must stress that,
 unlike in \cite{AMM},  our reproducing formula is not associate
with a semigroup with finite speed of propagation,  and therefore,
following  \cite{louMcintosh}, we are lead to work with a decomposition in molecules, replacing the usual decomposition in atoms of the functions in real Hardy spaces.}

 Classically, approximation on groups or manifolds can be done through the heat operator but the literature is rather poor on the properties of the heat kernel for the Rumin Laplacian. The primary goal of the present work is to fill in this gap and provide several ready-to-use properties of the heat equation on the Rumin complex. As an application, we use this heat kernel to prove a general Calder\'on reproducing formula on Rumin forms. Our contribution can then be seen as a further expansion of the  non-commutative harmonic analysis of differential complexes on the Heisenberg group. 
 
{In order to state our main results, we first recall some basic notations related to the Heisenberg group and the Rumin's complex of differential forms. The subsequent sections introduce all the necessary tools and the appendices expand on more details on the geometry and analysis on Heisenberg together with the Rumin's complex. We refer the reader to those for a more detail account. }

In this section, we present some basic notations and introduce both the structure of Heisenberg groups together with the formulation of the Rumin complex.  We denote by  $\he n$  the $(2n+1)$-dimensional Heisenberg
group, identified with $\rn {2n+1}$ through exponential
coordinates. A point $p\in \he n$ is denoted by
$p=(x,y,t)$, with both $x,y\in\rn{n}$
and $t\in\R$.
   If $p$ and
$p'\in \he n$,   the group operation is defined by
\begin{equation*}
p\cdot p'=(x+x', y+y', t+t' + \frac12 \sum_{j=1}^n(x_j y_{j}'- y_{j} x_{j}')).
\end{equation*}
{Notice that $\he n$ can be equivalently identified with $\mathbb C^{n}\times \mathbb R$
endowed with the group operation
$$
(z,t)\cdot (\zeta,\tau): = (z+\zeta, t+\tau - \frac12\,Im\,(z\bar{\zeta})).
$$
} 
The unit element of $\he n$ is the origin, that will be denote by $e$.
For
any $q\in\he n$, the {\it (left) translation} $\tau_q:\he n\to\he n$ is defined
as $$ p\mapsto\tau_q p:=q\cdot p. $$

    We denote by  $\mfrak h$
 the Lie algebra of the left
invariant vector fields of $\he n$. The standard basis of $\mfrak
h$ is given, for $i=1,\dots,n$,  by
\begin{equation*}
X_i := \partial_{x_i}- \frac12 y_i \partial_{t},\quad Y_i :=
\partial_{y_i}+\frac12 x_i \partial_{t},\quad T :=
\partial_{t}.
\end{equation*}
The only non-trivial commutation  relations are $
[X_{j},Y_{j}] = T $, for $j=1,\dots,n.$ 
The {\it horizontal subspace}  $\mfrak h_1$ is the subspace of
$\mfrak h$ spanned by $X_1,\dots,X_n$ and $Y_1,\dots,Y_n$:
${ \mfrak h_1:=\mathrm{span}\,\left\{X_1,\dots,X_n,Y_1,\dots,Y_n\right\}\,.}$

 Denoting  by $\mfrak h_2$ the linear span of $T$, the $2$-step
stratification of $\mfrak h$ is expressed by
\begin{equation*}
\mfrak h=\mfrak h_1\oplus \mfrak h_2.
\end{equation*}

\bigskip

{ 
The stratification of the Lie algebra $\mfrak h$ induces a family of non-isotropic dilations 
$\delta_\lambda: \he n\to\he n$, $\lambda>0$ as follows: if
$p=(x,y,t)\in \he n$, then
\begin{equation*}\label{dilations}
\delta_\lambda (x,y,t) = (\lambda x, \lambda y, \lambda^2 t).
\end{equation*}
}

Throughout this paper, we write also
\begin{equation*}\label{campi W}
W_i:=X_i, \quad W_{i+n}:= Y_i\quad { \mathrm{and} } \quad W_{2n+1}:= T, \quad \text
{for }i =1, \dots, n.
\end{equation*}

The dual space of $\mfrak h$ is denoted by $\covH 1$.  The  basis of
$\covH 1$,  dual to  the basis $\{X_1,\dots , Y_n,T\}$,  is the family of
covectors $\{dx_1,\dots, dx_{n},dy_1,\dots, dy_n,\theta\}$ where 
\begin{equation*}\label{theta}
 \theta
:= dt - \frac12\, \sum_{j=1}^n (x_jdy_j-y_jdx_j)
\end{equation*}
 is called the {\it contact
form} in $\he n$. 
We {also} denote by $\scalp{\cdot}{\cdot}{} $ the
inner product in $\covH 1$  that makes $(dx_1,\dots, dy_{n},\theta  )$ 
an orthonormal basis and
we set
\begin{equation*}
\omega_i:=dx_i, \quad \omega_{i+n}:= dy_i \quad { \mathrm{and} }\quad \omega_{2n+1}:= \theta, \quad \text
{for }i =1, \dots, n.
\end{equation*}

{ 
We put
$       \vetH 0 := \covH 0 =\R $
and, for $1\leq h \leq 2n+1$,
\begin{equation*}
\begin{split}
         \covH h& :=\mathrm {span}\{ \omega_{i_1}\wedge\dots \wedge \omega_{i_h}:
1\leq i_1< \dots< i_h\leq 2n+1\}.
\end{split}
\end{equation*}
We shall denote by $\Theta^h$ the basis of $ \covH h$ defined by
$$
\Theta^h:= \{ \omega_{i_1}\wedge\dots \wedge \omega_{i_h}:
1\leq i_1< \dots< i_h\leq 2n+1\}.
$$
The  inner product $\scal{\cdot}{\cdot}$ on $ \covH 1$ yields naturally an inner product 
$\scal{\cdot}{\cdot}$ on $ \covH h$
making $\Theta^h$ an orthonormal basis. The elements of $\cov h$ are identified with \emph{left invariant} differential forms
of degree $h$ on $\he n$.

The same construction can be performed starting from the vector
subspace $\mfrak h_1\subset \mfrak h$,
obtaining the {\it horizontal $h$-covectors} 

\begin{equation*}
\begin{split}
         \covh h& :=\mathrm {span}\{ \omega_{i_1}\wedge\dots \wedge \omega_{i_h}:
1\leq i_1< \dots< i_h\leq 2n\}.
\end{split}
\end{equation*}
It is easy to see that 
$$
\Theta^h_0 := \Theta^h \cap  \covh h
$$ 
provides an orthonormal
basis of $ \covh h$.

Keeping in mind that the Lie algebra $\mathfrak h$ can be identified with the
tangent space to $\he n$ at $x=e$, the neutral element el(see, e.g. \cite{GHL}, Proposition 1.72), 
starting from $\cov h$ we can define by left translation  a fiber bundle
over $\he n$  that we can still denote by $\cov h$. We can think of $h$-forms as sections of 
$\cov h$. We denote by $\Omega^h$ the
vector space of all smooth $h$-forms.

\bigskip

{As we stressed above,} the stratification
of the Lie algebra $\mfrak h$ yields a lack of homogeneity of de Rham's exterior differential
with respect to group dilations $\delta_\lambda$.  Thus, to keep into account the different degrees
of homogeneity of the covectors when they vanish on different layers of the
stratification, we introduce the notion of {\sl weight} of a covector as follows. This is at the core of Rumin construction of the differential complex.
}

\begin{definition}\label{weight} If $\eta\neq 0$, $\eta\in \covh 1$,  
 we say that $\eta$ has \emph{weight $1$}, and we write
$w(\eta)=1$. If $\eta = \theta$, we say $w(\eta)= 2$.
More generally, if
$\eta\in \covH h$, {  $\eta\neq 0$, }we say that $\eta$ has \emph {pure weight} $p$ if $\eta$ is
a linear combination of covectors $\omega_{i_1}\wedge\cdots\wedge\omega_{i_h}$
with $w(\omega_{i_1})+\cdots + w(\omega_{ i_h})=p$.
\end{definition}

The
following result holds
(see \cite{BFTT}, formula (16)):
\begin{equation*}\label{dec weights}
\covH h = \covw {h}{h}\oplus \covw {h}{h+1} =  \covh h\oplus \Big(\covh {h-1}\Big)\wedge \theta,
\end{equation*}
where $\covw {h}{p}$ denotes the linear span of the $h$-covectors of weight $p$ and a basis of
$ \covw {h}{p}$ is given by $\Theta^{h,p}:=\Theta^h\cap \covw {h}{p}$
(such a basis is usually called an adapted basis). {Consequently, the weight of a $h$-form
is either $h$ or $h+1$ and there are no {$h$-forms} of weight $h+2$, since  there
is only one 1-form of weight 2. 
}
Starting from  $\covw {h}{p}$, we can define by left translation  a fiber bundle
over $\he n$  that we can still denote by $\covw {h}{p}$. 
Thus, if we denote by  $\Omega^{h,p} $ the vector space of all
smooth $h$--forms in $\he n$ of  weight $p$, i.e. the space of all
smooth sections of $\covw {h}{p}$, we have
\begin{equation*}\label{deco forms}
\Omega^h = \Omega^{h,h}\oplus\Omega^{h,h+1} .
\end{equation*}

\bigskip

{ Starting from the notion of weight of a differential form, it is possible
to define a new complex of differential forms $(E_0^\bullet,d_c)$
that is homotopic to the de Rham's complex and respects the homogeneities
of the group. This is the Rumin's complex. 
}
 A crucial feature of $(E_0^\bullet,d_c)$ is that the ``exterior differential'' $d_c$ is an
 operator of order 1 with respect to group dilations when acting on forms of
 degree $h\neq n$, but of order 2 on $n$-forms.
 
 Following \cite{rumin_jdg}, we define
the operator $\Delta_{\he{},h}$  on $E_0^h$ by setting
\begin{equation*}
\Delta_{\he{},h}=
\left\{
  \begin{array}{lcl}
     d_cd^*_c+d^*_c d_c\quad &\mbox{if } & h\neq n, n+1;
     \\ (d_cd^*_c)^2 +d^*_cd_c\quad& \mbox{if } & h=n;
     \\d_cd^*_c+(d^*_c d_c)^2 \quad &\mbox{if }  & h=n+1.
  \end{array}
\right.
\end{equation*}
We point out that Rumin's Laplacian $\Delta_{\he{},h}$ is an operator of order
2 with respect to group dilations when acting on forms of
 degree $h\neq n$, but of order 4 on $n$-forms.

\subsection*{Main results}
Consider now the heat operator associated with the Rumin's Laplacian  $\Delta_{\he{},h} $ associated to the complex $(E_0^\bullet,d_c)$, i.e.
$$
\mathcal L:=  \partial_s +  \Delta_{\he{},h} \qquad\mbox{in ${\mathbb R_+}\times \he n$,}
$$
where $\partial_s$ {stands for} $\partial_s  I_d$,  $I_d$ being the identity $N_h\times N_h$ matrix. Our first result is 

\begin{theorem}
The operator $\mc L$ is hypoelliptic { on} ${\mathbb R_+}\times \he n$.
\end{theorem}

Building on the latter, we also prove the following basic properties of the heat kernel:

\begin{theorem}
{If $0\le h \le 2n+1$, }the operators $-\Delta_{\he{},h}:\mc D(\Delta_{\he{},h}) \subset L^2(\he n, E_0^h)\to L^2(\he n, E_0^h)$ are densely defined,
self-adjoint and dissipative, and therefore
generate strongly continuous analytic semigroup $\big( \exp(-s\Delta_{\he{}}) \big)_{s\ge 0}$ in $L^2(\he n, E_0^h)$.

Furthermore, there exists a  matrix-valued kernel 
\begin{equation*}\label{kernel existence}
h= h(s,p) =(h_{ij}(s,p))_{i,j=1,\dots, N_h} \in { \big(\mathcal D'(\he n)\big)^{N_h\times N_h}}
\end{equation*}
such that  
$$
\exp(-s\Delta_{\he{},h})\alpha= \alpha\ast h(s,\cdot) \qquad\mbox{for  { $\alpha\in\mc D(\he n, E_0^h)$.}}
$$

\end{theorem}

The kernel constructed in the previous statement has the following crucial properties. 

\begin{theorem}
We have:
\begin{enumerate}[i)]
\item $\mc L h = 0$ in $\big(\mc D'(\mathbb R_+\times \he n)\big)^{N_h\times N_h}$,
i.e.
\begin{equation*}\label{Aug 10 eq:1}
\Scal{\mc Lh}{A(s,y)}=0 \qquad\mbox{for all $A\in\mc D((0,\infty)\times \he n$,}
\end{equation*}
where the action of the heat operator $\mc L$ on $h$ 
{must be understood} as
the formal matrix product
\begin{equation*}\label{Lh def}
\mc L h: =\big(\mc L^{i,j}\big)_{i,j=1,\dots,N_h} \cdot\big( h_{i,j}\big)_{i,j=1,\dots,N_h},
\end{equation*}
{defined in the sense of distributions.}

\item the matrix-valued distribution  $h$ is smooth on 
$(0,\infty)\times \he n$.
%by Theorem \ref{hypoellipticity FS}.
In particular, if $\phi=\sum_j \phi_{j} \in \mc D(\mathbb R_+\times \he n, E_0^h)$, we can write
\begin{equation*}\label{h smooth}
\Scal{h}{\phi} =\sum_{i,j}  \big( \int_{\mathbb R_+\times \he n} h_{i,j}(s,p)\phi_{j}(s,p)\, ds\, dp\big) \xi_i;
\end{equation*}
\item if $r>0$ 
\begin{equation*}\label{jan11 eq:3}
h (r^as,y) = r^{-Q} h(s, \delta_{1/r}y)\qquad\mbox{for $s>0$ and $y\in \he n$;}
\end{equation*}

\end{enumerate}

\end{theorem}

We now finally state how we use the heat kernel to build a reproducing formula.  Denote $\alpha\in L^1(\he n, E_0^h)$, such that  $d_c\alpha=0$ and define the map
\begin{equation*}\label{Aug13 eq:1}
F(s,x):= d_c^*\big(h(\frac{s}{2}, \cdot)\ast \alpha\big)(x)\qquad s>0.
\end{equation*}

We then have 

\begin{theorem}
If $\alpha \in L^1(\he n,E_0^h)$ as above, we have:
\begin{equation*}\label{reproducing}
\alpha = -\int_0^\infty d_c \big(h(\frac{s}{2}, \cdot)\ast F(s,\cdot) \big)\, ds
\end{equation*}
\end{theorem}

\subsection*{Notations} We refer the reader to the Appendices for notations which are used in the paper.

\medskip

The paper is organized as follows: in Section 2, we introduce currents on Heisenberg groups. Section 3 is the core of the paper and is devoted to a thorough investigation of the heat kernel associated to the Rumin complex and its application to the reproducing formula. In the subsequent appendices, we recall the necessary tools from the construction of Rumin and the Heisenber groups (Appendix A) and from the analysis on groups as developed by Folland and Stein (Appendix B).

\section{Currents on Heisenberg groups}\label{currents}
\bigskip

Let  $U\subset \he n$ be an open set. We shall use the following classical notations:
$\mc E(U)$ is the space of all smooth function on $U$,
and $\mc D(U)$  is the space of all compactly supported smooth functions on $U$,
endowed with the standard topologies (see e.g. \cite{treves}).
The spaces $\mc E'(U)$ and $\mc D'(U)$ are their
dual spaces of distributions.

\begin{definition}\label{corrente in gen}
 If $\Omega\subset\he n$ is an open set,
we say that $T$ is a  $h$-current on $\Omega$
if $T$ is a continuous linear functional on $\mc D(\Omega, E_0^h)$
endowed with the usual topology. We write $T\in \mc D'(\Omega, E_0^h)$. The definition of $\mc E'(\Omega, E_0^h)$ is given
analogously.
\end{definition}

If $T\in\mathcal D'(\Omega)$ and $\phi\in\mathcal D(\Omega)$, we
shall denote the action of $T$ on $\phi$ by $\Scal{T}{\phi}$. An analogous notation will be
used for currents versus differential forms.

\begin{proposition}\label{corrente by distribuzione}
If $\Omega \subset \he n$ is an open set, and
$T\in\mathcal D'(\Omega)$ is a (usual) distribution, then
$T$ can be identified canonically with a $n$-current $\tilde
T \in  \mc D'(\Omega, E_0^n)$ through the formula
\begin{equation}\label{cbd}
\Scal{\tilde T}{\alpha}:=\Scal{T}{\ast \alpha}
\end{equation}
for any $\alpha\in \mc D(\Omega, E_0^n)$.
Reciprocally, by (\ref{cbd}), any $n$-current $\tilde T$
can be identified with an usual distribution $T \in D'(\Omega)$.
\end{proposition}

\begin{proof}
See \cite{dieudonne}, Section 17.5, and \cite{BFT1},
Proposition 4.
\end{proof}

Following \cite{federer}, 4.1.7, we give the following definition.
\begin{definition}
If $T\in \mc D'(\Omega, E_0^n)$, and $\phi \in 
\mc E(\Omega, E_0^k)$, with $0\le k\le n$, we define $T\res
\phi\in \mathcal \mc D'(\Omega, E_0^{n-k})$ by the identity
$$
\Scal{T\res \phi}{\alpha}:=\Scal{T}{ \alpha\wedge\phi}
$$
for any $\alpha \in \mc D(\Omega, E_0^{n-k})$.

%We notice that, following \cite{dieudonne},
% an alternative notation for $T\res \phi$
% could be $T\Wedge \phi$.
\end{definition}

The following result is taken from \cite{BFT1}, Propositions 5 and 6,
and Definition 10, but we refer also to \cite{dieudonne}, Sections 17.3
17.4 and 17.5.

\begin{proposition}\label{corrente in coordinate}
Let $\Omega \subset \he n$ be an open set.
If $1\le h\le n$,  $N_h= \mathrm{dim}\, E_0^h$ and 
$\Xi_0^h=\{\xi_1^h,\dots \xi^h_{N_h}\}$ is a left invariant  
basis
of $E_0^h$ and $T\in \mc D'(\Omega, E_0^{h})$, then
\begin{enumerate}
\item[i)] there
exist (uniquely determined)
$T_1,\dots,T_{N_h}\in \mathcal D'(\Omega)$ such that we
can write
$$
T=\sum_j\tilde T_j\res (\ast\xi^h_j);
$$
%Notice that $\ast\xi_j\Wedge\alpha$ is well defined for any
%$\alpha\in   \mc D_{\he{}}^{m}(\mathcal U)$, for $\ast\xi_j\in
%\mc E_{\he{}}^{2n+1-m}(\mathcal U)$, and the pair
%$(2n+1-m,m)$ is always wedge--admissible, since
%$2n+1-m$ and $m$ can not be both less than $n$;
\item[ii)] if
$\alpha\in \mathcal E(\Omega, E_0^{h})$, then $\alpha$ can be
identified canonically with a $h$-current $T_\alpha$ through the
formula
\begin{equation}\label{form=current}
\Scal{T_\alpha}{\beta}:=\int_{\Omega}\ast\alpha\wedge\beta
\end{equation}
for any $\beta\in \mc D(\Omega, E_0^{h})$.
%Notice that $\ast \alpha \Wedge \beta$ is well defined.
Moreover, if
$\alpha=\sum_j\alpha_j\xi_j^h$ then
$$
T_\alpha=\sum_j\tilde\alpha_j\res(\ast\xi_j^h);
$$
\item[iii)] we
say that $T $ is smooth
in $\mathcal U$ when $T_1,\dots,T_{N_h}$
are (identified with) smooth functions.
This is clearly equivalent to say that there exists $\beta\in
\mathcal E(\Omega, E_0^{h})$ such that
 $$
 \Scal{T}{\alpha}=\int_{\mc U}\scal{\beta}{\alpha}\,dV
 $$
 for any $\alpha\in  \mc D(\Omega, E_0^{h})$
 (in fact, we choose $\beta=\sum_jT_j\xi_j^h$).
\end{enumerate}
\end{proposition}

%\begin{remark}\label{current simply}
%By analogy with \eqref{form=current}, if $T_1,\dots,T_{N_h}\in \mathcal D'(\Omega)$
%and
%$$
%T=\sum_j\tilde T_j\res (\ast\xi^h_j),
%$$
%we  write (though incorrectly)
%$$
%T=\sum_j T_j \xi_j.
%$$
%\end{remark}

{
\begin{remark} If $1\le h\le n$,  let 
\begin{equation*}
\Xi_0^h=\{\xi_1^h,\dots \xi^h_{N_h}\}
\end{equation*}
 be a left invariant  
{basis 
%$\Xi_0^h:=\{\xi_1^h,\dots,\xi_{\N_h}\}$}
 of $E_0^h$}. Then the linear maps on $E_0^h$
$$
\alpha  \to (\xi_j^h)^*(\alpha) := \ast (\alpha\wedge \ast \xi_j^h)
$$
belong to $(E_0^h)^*$ (the dual of $E_0^h$) and
$$
(\xi_j^h)^*(\xi_i^h) = \ast  (\xi_i^h \wedge \ast \xi_j^h) = \delta_{i,j} \ast dV = \delta_{i,j},
$$
i.e. $(\Xi_0^h)^*=\{(\xi_1^h)^* ,\dots, (\xi^h_{N_h})^* \}$ is a left invariant  
dual basis of $(E_0^h)^*$.

\end{remark}

\begin{remark}\label{current simply}

Let us remind the notion of distribution section of a finite-dimensional
vector bundle $\mc F$: a distribution section is a continuous linear
map on the {space of} compactly supported sections of the dual vector bundle
$\mc F^*$ (see, e.g., \cite{treves}, p. 77). 

Let $T$ be a current on $E_0^h$, 
$$
T=\sum_j\tilde T_j\res (\ast\xi^h_j),
$$
where $T_1,\dots,T_{N_h}\in \mathcal D'(\Omega)$.
Then $T$ can be seen as a section of $(E_0^h)^*$.
Indeed, if $\alpha=\sum_i \alpha_i \xi_i^h\in \mc D(\Omega, E_0^h)$
\begin{equation*}\begin{split}
\Scal{T}{\alpha} & = \sum_{j} \Scal{\tilde T_j\res  (\ast\xi_j^h)}{\alpha} 
= \sum_{j} \Scal{\tilde T_j}{ \alpha\wedge (\ast\xi_j^h)}
\\&
=  \sum_{j} \Scal{T_j}{ \alpha_j} =  \sum_{i,j} \Scal{T_j}{ (\xi_j^h)^*(\alpha_i\xi_i^h)}
 = \sum_j   \Scal{T_j }{(\xi_j^h)^*(\alpha)},
 \end{split}\end{equation*}
 where the dualities in the first line are meant as dualities between currents
 and test forms, where the dualities in the second line are meant as dualities {between}
 distributions and test functions.
Thus we can write
formally 
\begin{equation}\label{Aug11 eq:4}
T=\sum_j T_j (\xi_j^h)^*
\end{equation}
and we can identify $T$ with a vector-valed distribution $(T_1,\dots,T_{N_h})$.

We notice also that, if $\alpha=\sum_j\alpha_j \xi_j\in \mc E(\Omega, E_0^h)$,
then
$$
T_\alpha=\sum_j \alpha_j (\xi_j^h)^*.
$$
\end{remark}
 }
 
 %%%

%$\mc L h = 0$ in $\big(\mc D'(\mathbb R_+\times \he n)\big)^{N_h\times N_h}$,
%where the action of the heat operator $\mc L$ on $h$ has to be meant as
%the formal matrix product
%\begin{equation}\label{Lh def}
%\mc L h: =\big(\mc L^{i,j}\big)_{i,j=1,\dots,N_h} \cdot\big( h_{i,j}\big)_{i,j=1,\dots,N_h}.
%\end{equation}
%
%the matrix-valued distribution  $h$ is smooth on 
%$(0,\infty)\times \he n$.
%%by Theorem \ref{hypoellipticity FS}.
%In particular, if $\phi=\sum_j \phi_{j} \in \mc D(\mathbb R_+\times \he n, E_0^h)$, we can write
%\begin{equation}\label{h smooth}
%\Scal{h}{\phi} =\sum_{i,j}  \big( \int_{\mathbb R_+\times \he n} h_{i,j}(s,p)\phi_{j}(s,p)\, ds\, dp\big) \xi_i;
%\end{equation}

%{\color{green}
%\begin{definition}\label{matrix distribution*}
%If $T_{i,j} \in \mc D'(\he n)$ for $i,j=1,\dots,N_h$, we shall refer to the matrix
%$T:=\big( T_{i,j}\big)_{i,j=1,\dots,N_h}$ as to a matrix-valued distribution.
%
%Following e.g. \cite{treves1}, p. 76, a matrix-valued distribution can be identified with:
%\begin{itemize}
%\item[i)] a section of the fiber bundle $\big(\he n,\mathrm{Hom}\, (E_0^h,E_0^h)\big)$
%{\color{blue}with the} usual projection{\color{blue}.}
%
%\end{itemize}\footnote{To omit?}
%
%\end{definition}
%}

 {
 \begin{definition}\label{matrix distribution}

If $T_{i,j} \in \mc D'(\he n)$ for $i,j=1,\dots,N_h$, we shall refer to the matrix
$T:=\big( T_{i,j}\big)_{i,j=1,\dots,N_h}$ as to a matrix-valued distribution
$$
\big( T_{i,j}\big)_{i,j=1,\dots,N_h}: \mc D(\he n, E_0^h) \to E_0^h
$$
{defined} through the identity
\begin{equation}\label{ambiguo}
\Scal{T}{\alpha} := \sum_i  \sum_j \Scal{T_{i,j}}{\alpha_j} \xi_i
\end{equation}
if $\alpha=\sum_j\alpha_j\xi_j$.

 A matrix-valued distribution $T:=\big( T_{i,j}\big)_{i,j=1,\dots,N_h}$
can also be seen as a distribution section of the fiber bundle
$(\he n, E_0^h\otimes (E_0^{h })^*)$ (see, e.g. \cite{treves1}, p. 76) through
the action
$$
\Scal{T}{A} = \sum_{i,j} \Scal{T_{i,j} }{A_{i,j}}
$$
for $A:= \sum_{i,j} A_{i,j} \xi_i^* \otimes \xi_j \in \mc D(\he n, (E_0^h)^* \otimes E_0^{h })$.

As we did  in Remark \ref{current simply} we can write
$$
T= \sum_{i,j} T_{i,j} \xi_i\otimes \xi_j^*.
$$
\end{definition}

}

\section{Rumin's Laplacian and heat operator in $E_0^\bullet$}

This section is our main contribution. After a brief introduction about the Rumin Laplacian, we derive several basic properties of the associated heat operator. The last section is then devoted to an application to the construction of a Calder\'on formula in this setting. 
\subsection{Rumin's Laplacian and its fundamental solution}

 \begin{definition}\label{rumin laplacian} 
In $\he n$, following \cite{rumin_jdg}, we define
the operator $\Delta_{\he{},h}$  on $E_0^h$ by setting
\begin{equation*}
\Delta_{\he{},h}=
\left\{
  \begin{array}{lcl}
     d_cd^*_c+d^*_c d_c\quad &\mbox{if } & h\neq n, n+1;
     \\ (d_cd^*_c)^2 +d^*_cd_c\quad& \mbox{if } & h=n;
     \\d_cd^*_c+(d^*_c d_c)^2 \quad &\mbox{if }  & h=n+1.
  \end{array}
\right.
\end{equation*}

\end{definition}

Notice that $-\Delta_{\he{},0} $ is the usual positive sub-Laplacian of
$\he n$.

%From now on, to avoid cumbersome notations, whenever there is no risk of misunderstandings
%we shall drop the index $h$ and we shall write $\Delta_\he{}$ for $\Delta_{\he{},h}$.

\begin{definition}[Laplacian of a current] In the sequel, when  $T$ is is a $h$-current identified with its {components} $(T_1,\dots,T_{N_h})$ with respect to a
fixed basis $(\xi_1)^*,\dots,(\xi_{N_h})^*$ of $(E_0^h)^*$ as in Remark \ref{current simply}, il will be useful to think of $\Delta_{\he{},h}$
as of a matrix-valued differential operator $\big(\Delta_{\he{},h}^{ i,j}\big)_{i,j=1,\dots,N_h}$
acting as follows (again with the notations of Remark \ref{current simply}):
\begin{equation}\label{laplace matrix}
\Delta_{\he{},h} T = \Delta_{\he{},h} (\sum_j T_j(\xi_j)^*) = \sum_{i,j} \big( \Delta_{\he{},h}^{ i,j}T_j \big)(\xi_i)^*.
\end{equation}
\end{definition}

It is easy to see that
\begin{lemma}\label{symmetry laplace} If the basis  $\xi_1,\dots,\xi_{N_h}$ is orthonormal
with respect to the scalar product used to define $d_c^*$, then
\begin{equation}\label{adjoint}
(\Delta_{\he{},h}^{ i,j})^* = \Delta_{\he{},h}^{ j,i},
\end{equation}
where $(\Delta_{\he{},h}^{ i,j})^*$ is the formal adjoint of $\Delta_{\he{},h}^{ i,j}$ {on} $\mc D(\he n)$.
\end{lemma}

%\begin{definition}\label{matrix distribution}
%If $T_{i,j} \in \mc D'(\he n)$ for $i,j=1,\dots,N_h$, we shall refer to the matrix
%$T:=\big( T_{i,j}\big)_{i,j=1,\dots,N_h}$ as to a matrix-valued distribution
%$$
%\big( T_{i,j}\big)_{i,j=1,\dots,N_h}: \mc D(\he n, E_0^h) \to E_0^h
%$$
%through the identity
%\begin{equation}\label{ambiguo}
%\Scal{T}{\alpha} := \sum_i  \sum_j \Scal{T_{i,j}}{\alpha_j} \xi_i
%\end{equation}
%if $\alpha=\sum_j\alpha_j\xi_j$.
%
%{\color{red} A matrix-valued distribution $T:=\big( T_{i,j}\big)_{i,j=1,\dots,N_h}$
%can also be seen as a distribution section of the fiber bundle
%$(\he n, E_0^h\otimes (E_0^{h })^*)$ (see,e.g. \cite{treves}, p. 76) through
%the action
%$$
%\Scal{T}{A} = \sum_{i,j} \Scal{T_{i,j} }{A_{i,j}}
%$$
%for $A:= \sum_{i,j} A_{i,j} \xi_i^* \otimes \xi_j \in \mc D(\he n, (E_0^h)^* \otimes E_0^{h })$.
%
%As we did (incorrectly) in Remark \ref{current simply} we can write
%$$
%T= \sum_{i,j} T_{i,j} \xi_i\otimes \xi_j^*.
%$$
%
%}
{
\begin{definition}[Laplacian of matrix-valued distribution] If $T=\big( T_{i,j}\big)_{i,j=1,\dots,N_h}$ is a  matrix-valued distribution, we shall 
denote by $\Delta_{\he{},h} T$ the matrix-valued distribution defined by

\begin{equation}\label{laplace of a matrix ter}
\begin{split}
\Scal{\Delta_{\he{},h}T}{ A} & 
:= 
\sum_{i,j,\ell} \Scal{T_{i,j}}{\Delta_{\he{},h}^{i, \ell} A_{\ell, j}}
=
\sum_{i,j,\ell} \Scal{\Delta_{\he{},h}^{\ell, i}T_{i,j}}{A_{\ell, j}}
\end{split}\end{equation}
for all test matrices $A = (A_{\ell, j})$. 
%\end{definition}

\begin{remark}
We stress that the notation $\Scal{\Delta_{\he{},h}T}{ A}$
may conflict with the notation $\Scal{T}{\alpha}$ of \eqref{ambiguo} if $\alpha=\sum_j\alpha_j\xi_j$
is {a} test form. If there is no way to misunderstanding, we shall use this ambiguous notation,
using Greek lower case characters for forms and capital Latin characters for matrices.

\end{remark}

%\begin{equation}\label{laplace of a matrix bis}
%\Scal{\Delta_{\he{},h} T}{\alpha} := \Scal{T}{\Delta_{\he{},h} \alpha} 
%\end{equation}
%for all test form $\alpha\in  \mc D(\he n, E_0^h)$.
%
%Notice that
%\begin{equation}\label{laplace matrix explicit}\begin{split}
%\Scal{T}{\Delta_{\he{},h} \alpha} & =
%\sum_{i,j,\ell} \Scal{T_{i,\ell}}{\Delta_{\he{},h}^{\ell,j}\alpha_j}\xi_i
%= \sum_{i,j,\ell} \Scal{\Delta_{\he{},h}^{j,\ell}T_{i,\ell}}{\alpha_j}\xi_i
%\end{split}\end{equation}
%
%}
%
%\begin{equation}\label{laplace of a matrix}
% \big(\Delta_{\he{},h} T\big)_{i,j}:= \sum_\ell \Delta_{\he{},h}^{i,\ell} T_{\ell,j}.
%\end{equation}

In addition, if $T=\big( T_{i,j}\big)_{i,j=1,\dots,N_h}$ and $S=\big( S_{i,j}\big)_{i,j=1,\dots,N_h}$
are  matrix-valued distributions, then the convolution $T\ast S$ is defined by
\begin{equation}\label{T ast S}
(T\ast S)_{i,j=1,\dots, N_h} := \sum_\ell T_{i,\ell}\ast S_{\ell,j},
\end{equation}
provided all convolutions in \eqref{T ast S} are well defined.
\end{definition}

}

%\textcolor{green}{Denote by $\Delta_{\he n}$ Rumin's Laplacian of $m$-forms. Then there exist left invariant
%homogeneous differential operators
%$P_{i,j}$
%of order 2 or 4 according to the degree of the forms such that, if $\alpha=\sum_j \alpha_j \xi_j$
%then
%$$
%\Delta_{\he n}\alpha = \sum_i\sum_j P_{ij}\alpha_i\, \xi_j.
%$$
%Thus, if $\phi \in \mc D(\R)$, then
%$$
%\Delta_{\he n}(\phi\ast\alpha) =  \phi \ast \Delta_{\he n}\alpha.
%$$
%Indeed
%\begin{equation}\begin{split}
%\Delta_{\he n}(\phi\ast\alpha)& = \sum_i\sum_j P_{ij}(\phi\ast\alpha_j)\, \xi_i
%= \sum_i\sum_j (\phi\ast P_{ij}\alpha_j)\, \xi_i 
%\\&
%= 
%\sum_i \phi\ast \big( \sum_j P_{ij}\alpha_j\big) \, \xi_i 
%=  \phi \ast \Delta_{\he n}\alpha.
%\end{split}\end{equation}
%In other words, $\Delta_{\he n}$ can be identified with a $N_m\times N_m$-matrix valued differential operator $P:=\big(P_{ij} \big)_{i,j = 1\dots M_h}$
%defined on $\big(W^{2,2}(\he n)\big)^{N_m}$ if $m\neq n$ and $\big(W^{4,2}(\he n)\big)^{N_m}$ if $m=n$.
%}

\begin{theorem}[see \cite{BFT3}, Theorem 3.1] \label{global solution}
If $0\le h\le 2n+1$, then the differential operator $\Delta_{\he{},h}$ is
homogeneous of degree $a$ with respect to { the} group dilations, where $a=2$ if $h\neq n, n+1$ and  $a=4$ 
if $h=n, n+1$. We have:

\begin{enumerate}[i)]

\item By \cite{rumin_jdg}, \cite{ponge_mams}, $\Delta_{\he{},h}$ is a Rockland operator and
hence  is maximal hypoelliptic (in particular hypoelliptic), in the sense of \cite{HN},
i.e., if $\Omega\subset\he n$ is a bounded open set, then there
exists $C=C_\Omega$ such that for any $p\in (1,\infty)$ and
for any multi--index
$I$ with $|I|=a$
we have
\begin{equation}\label{HN ex p=2}
\|W^I\alpha\|_{L^{ p}(\he n, E_0^h)}\le C
\left(
\|\Delta_{\he{},h}\alpha\|_{L^{p}(\he n, E_0^h)}+\|\alpha\|_{L^{p}(\he n, E_0^h)}
\right)
\end{equation}
for any $\alpha\in \mc D (\Omega, E_0^h)$ and where $W^I$ are defined in \eqref{campi W}.

\item {For} $j=1,\dots,N_h$ there exists
\begin{equation}\label{numero}
    K_j =
\big(K_{1j},\dots, K_{N_h j}\big), \quad j=1,\dots N_h
\end{equation}
 with $K_{ij}\in\mc D'(\he{n})\cap \mc
E(\he{n} \setminus\{0\})$,
$i,j =1,\dots,N$
{ such that $\sum_\ell \Delta_{\he{},h}^{i,\ell}K_{\ell,j} = 0$
if $i\neq j$ and  {$\sum_\ell \Delta_{\he{},h}^{i,\ell}K_{\ell,i} = \delta_e$ (where $\delta_e$ denotes the Dirac mass at $p=e$)}};
\item {If} $a<Q$, then the $K_{ij}$'s are
kernels of type $a$ in the sense of Definition \ref{folland kernels} (and hence belong to $\mathbf K^{a-Q}$ in the sense
of Definition \ref{Kalpha})
 for
$i,j
=1,\dots, N_h$. { In particular the $K_{i,j}$ are tempered distributions.}
 If $a=Q$,
then the $K_{ij}$'s satisfy the logarithmic estimate
$|K_{ij}(p)|\le C(1+|\ln\rho(p)|)$ and hence
belong to $L^1_{\mathrm{loc}}(\he{n})$. 
Moreover, their horizontal derivatives  $W_\ell K_{ij}$,
$\ell=1,\dots,2n$, are
kernels of type $Q-1$. In particular,
the $K_{ij}$'s belong to $S'(\he n)$ for $a\le Q$ for $i,j=1,\dots,N_h$;
\item {When} $\alpha = \sum_j \alpha_j\xi_j \in
\mc D(\he{n},E_0^h)$,
if we set
\begin{equation}\label{numero2}\begin{split}
   \Delta_{\he{},h}^{-1}\alpha : &= 
      \sum_{i,j} \big(\alpha_j\ast  K_{i,j}\big) \xi_i
      \\&
      =\alpha\ast (K_{i,j})_{i,j}\qquad\mbox{(see Definition \ref{alpha ast matrix}),}
\end{split}\end{equation}
 then 
 \begin{equation}\label{i}
\Delta_{\he{},h}\Delta_{\he{},h}^{-1}\alpha =  \alpha.
\end{equation}
 
Moreover, if $a<Q$, also $\Delta_{\he{},h}^{-1}\Delta_{\he{},h}\alpha =\alpha$.
Thus, if we identify the operator $\Delta_{\he{},h}^{-1}$ with its distributional
kernel, we can write
$$
 (K_{i,j})_{i,j}=:  \Delta_{\he{},h}^{-1}.
$$
{With the notation of} \eqref{laplace of a matrix ter},
\eqref{i} can be written as
\begin{equation}\label{numero 4}
\Delta_{\he{},h}\Delta_{\he{},h}^{-1} = \delta_{e, h},
\end{equation}
where $\delta_{e, h}$
 is the matrix-valued distribution $(a_{i,j})_{i,j=1,\dots,N_h}$ where
$a_{i,j} = 0$ if $i\neq j$, and {$a_{i,i}=\delta_e$} for $i=1,\dots,N_h$, so that
$$
\delta_{e, h}u = u(0)\qquad\mbox{for all $u\in \mc D(\he n,E_0^h)$;}
$$

\item {If} $a=Q$, then for any $\alpha\in
\mc D(\he{n},\rn {N_h})$ there exists 
$\beta_\alpha:=(\beta_1,\dots,\beta_{N_h})\in \rn{N_h}$,  such that
\begin{equation}\label{numero3}
\Delta_{\he{},h}^{-1}\Delta_{\he{},h}\alpha - \alpha = \beta_\alpha.
\end{equation}
This situation arises only when $n=1$ and $h=1,2$.
\end{enumerate}
\end{theorem}

%{   \begin{remark}\label{pavidi}
%If $a<Q$, 
% $ \Delta_{\he{},h} (\Delta_{\he{},h}^{-1} - \ccheck \Delta_{\he{},h}^{-1}) = 0$ and hence $\Delta_{\he{},h}^{-1} = \ccheck \Delta_{\he{},h}^{-1}$,
%by the Liouville-type theorem of \cite{BFT3}, Proposition 3.2.
%\end{remark}

The following vector-valued Liouville type theorem has been proved in \cite{BFT3}, Proposition 3.2.
\begin{proposition}\label{liouville}
Suppose  $\mc L$ is a left-invariant hypoelliptic  differential
operator  which is formally self-adjoint. Suppose also that $\mc L$ is
homogeneous of degree {$a\le Q$}. 
If $T=(T_1,\dots,T_N)\in \mc S'(\he{n})^N$ satisfies $\mc L T=0$, then $T $
is a (vector--valued) polynomial.

In particular, by Theorem \ref{global solution}, i), the {proposition} applies to $\mc L = \Delta_{\he{},h}$.
\end{proposition}

As a consequence the following results can be proved as in \cite{BLU}, Propositions 5.3.10 and 5.3.11.

\begin{theorem}\label{symmetries} Suppose {$Q>a$}. We have:
\begin{itemize}
\item[i)] if $\tilde K:=  \big(\tilde K_{i,j}\big)_{i,j}$ with $ \tilde K_{i,j}\in \mc S'(\he n)\cap\mc E(\he n\setminus \{e\})$, $i,j=1,\dots, N_h$,
 vanishes at infinity and satisfies \eqref{numero 4}, then $\tilde K= \Delta_{\he{},h}^{-1}$;
\item[ii)] $\Delta_{\he{},h}^{-1} = \ccheck \Delta_{\he{},h}^{-1}$ (identity among convolution kernels).
\end{itemize}

\end{theorem}

\begin{proof} 
Let us prove i). Set $\Gamma:= \tilde K - \Delta_{\he{},h}^{-1}$. By Theorem \ref{global solution}, iii) $\Gamma$ belongs to $L^1_{\mathrm{loc}}$.
In addition, $\Delta_{\he{},h}\Gamma=0$, so that,
by Proposition \ref{liouville}, $\Gamma$ is a vector-valued polynomial. But, by Theorem \ref{global solution}, iii),
$\Gamma$ has at most a logarithmic behavior at infinity and hence vanishes.

Let us prove ii).  Take $\phi = \sum_j \phi_j\xi_j\in \mc D(\he n,E_0^h)$, and set
$$
{u(p):= \Delta_{\he{},h}\phi \ast  \Delta_{\he{},h}^{-1}(p) =
\sum_k \sum_{j,\ell} \Big(
\int\Delta^{\ell,j}\phi_j(q)K_{\ell,k}(q^{-1}p)dq\Big)\xi_k}.
$$
Arguing on the entries, it turns out that the matrix-valued distribution $u$ is
well defined and smooth. In addition, if $h\neq n,n+1$, by  Lemma \ref{pointwise}
\begin{equation}\label{oct10:4}
u(p)=   O(|p|^{2-Q})\quad\mbox{as }p\to\infty.
\end{equation}
Analogously, if $h=n,n+1$ and $n>1$
\begin{equation}\label{oct10:5}
u(p) = O(|p|^{4-Q})\quad\mbox{as }p\to\infty,
\end{equation}
and, eventually,
\begin{equation}\label{oct10:6}
u(p) = O(\ln |p| )\quad \mbox{as }p\to\infty
\end{equation}
when $n=1$ and $h=1,2$. Take now $\psi =\sum_j \psi_j\xi_j\in \mc D(\he n,E_0^h)$. We have
\begin{equation}\begin{split}\label{oct10:2}
\int & \scal{u}{\Delta_{\he{},h}\psi}\, dp =
\sum_{k,i} \sum_{j,\ell}\int\Big({\int\Delta^{\ell,j}_{\he{},h}\phi_j(q)K_{\ell,k}(q^{-1}p)dq}\Big)  \Delta^{k,i}_{\he{},h}\psi_i(p) \, dp
\\
&= \sum_{j,\ell}\int\, dq \, \Delta^{\ell,j}_{\he{},h}\phi_j(q)\int\, dp \sum_{k,i} K_{\ell,k}(q^{-1}p) \Delta^{k,i}_{\he{},h}\psi_i(p)
\\
&= \sum_{j,\ell}\int\, dq \, \Delta^{\ell,j}_{\he{},h}\phi_j(q)\int\, dp \sum_{k} K_{\ell,k}(q^{-1}p) \big(\Delta_{\he{},h} \psi\big)_k(p).
\end{split}\end{equation}
Now, putting $q^{-1}p=:\eta$, and keeping in mind that $\Delta_{\he{},h}$ is left-invariant,
\begin{equation}\begin{split}\label{oct10:1}
\int\, & dp\, \sum_{k} K_{\ell,k}(q^{-1}p) \big(\Delta_{\he{},h} \psi\big)_k(p)
=  \int\,  d\eta\, \sum_{k} K_{\ell,k}(\eta ) \big(\Delta_{\he{},h} \psi\big)_k(\tau_q \eta)
\\
& =  \int\,  d\eta\, \sum_{k} K_{\ell,k}(\eta ) \big(\Delta_{\he{},h}( \psi\circ \tau_q) \big)_k( \eta)
\\
& =   \big(\Delta_{\he{},h}^{-1} \Delta_{\he{},h}{(\psi\circ\tau_q)} \big)_\ell (e) = \psi_\ell (q),
\end{split}\end{equation}
by Theorem \ref{global solution}, iv), provided $h\neq 1,2$ if $n=1$. If $n=1$ and $h=1,2$, the
last line must be replaced by
$$
\psi_\ell (q) + \beta_\ell,
$$
where  $\sum_k\beta_k \xi_k$ is a constant coefficients form (depending on $\psi$).

{Plugging \eqref{oct10:1} in \eqref{oct10:2}}, we get
\begin{equation*}\begin{split}%\label{oct10:2}
\int & \scal{u}{\Delta_{\he{},h}\psi}\, dp = \sum_{j,\ell}\int\, dq \, \Delta^{\ell,j}_{\he{},h}\phi_j(q) \psi_\ell (q)=
\scal{\Delta_{\he{},h}\phi}{\psi},
\end{split}\end{equation*}
i.e. 
$\Delta_{\he{},h}\phi=\Delta_{\he{},h}u$ in the sense of distributions. On the other hand,
$$
\scal{\Delta_{\he{},h}\phi}{\beta}=0,
$$
and the conclusion still holds when $n=1$ and $h=1,2$.

We can conclude that $\Delta_{\he{},h}( \phi-u)=0$, so that, by Proposition \ref{liouville}.
$\phi -u$ is a polynomial form. On the other hand, by \eqref{oct10:6}, $\phi-u$ has
at
most a logarithmic behavior at infinity, so that $\phi-u$=0. In particular,
\begin{equation*}\begin{split}
\phi(e) & = u(e)= \sum_k \sum_{j,\ell} \Big( \Delta^{\ell,j}\phi_j(q)K_{\ell,k}(q^{-1})\Big)\xi_k
\\
& = \phi\ast \ccheck \Delta_{\he{},h}^{-1}.
\end{split}\end{equation*}
{Since $\ccheck \Delta_{\he{},h}^{-1}$ satisfies the assumptions of i), we obtain
$\ccheck \Delta_{\he{},h}^{-1}=\Delta_{\he{},h}^{-1}$.}

\end{proof}

\medskip

The aim of the following result is the characterization of some integer order Sobolev spaces of
forms in $E_0^\bullet$ in terms of integer powers of Rumin's Laplacian. More
precisely, we prove that

\begin{proposition}\label{equivalent norm}  If $k\in \mathbb N$, $\alpha\in  L^2 (\he n,E_0^h) \cap \mc D(\Delta_{\he{},h}^k, E_0^h)$, then
$$
\| \alpha\|_{L^2(\he n,E_0^h)} + \| \Delta_{\he{},h}^k\alpha \|_{L^2(\he n,E_0^h) }
$$
is equivalent to the norm of $\alpha$ in $W^{ak,2}(\he n,E_0^h)$ (we remind that 
$a=2$ if $h\neq n, n+1$ and $a=4$ if $h=n,n+1$).

\end{proposition}
 \begin{proof} For sake of simplicity we take $n>1$. The case $n=1$
 can be handled in the same way.

Obviously, we have just to show that
 $$
 \| \alpha\|_{L^2 (\he n, E_0^h)} + \| \Delta_{\he{},h}^k \alpha\|_{L^2 (\he n, E_0^h) } \ge c\, \|\alpha\|_{W^{ak,2}(\he n, E_0^h)}.
 $$
 Suppose first $\alpha \in \mc S_0(\he n,E_0^h)$. By Proposition \ref{s0}, $\Delta_{\he{},h}^{-1}: \mc S_0(\he n,E_0^h)\to \mc S_0(\he n,E_0^h)$,
 so that we can write
 $$
 \alpha  =  (\Delta_{\he{},h}^{-1})^k\circ  \Delta_{\he{},h}^k \alpha .
 $$
Notice now that, by Proposition \ref{composizione} and Theorem \ref{global solution}, i)
$$
(\Delta_{\he{},h}^{-1})^k = \mc O_0(K_k), \qquad\mbox{with $K_k\in \mathbf K^{ak -Q}$.}
$$
Moreover, by Lemma \ref{deriv}, if $ d(I)=ak$
 $$
 X^I (\Delta_{\he{},h}^{-1})^k  = \mc O_0(X^I K_k), \qquad\mbox{ where $X^IK_k \in \mathbf K^{-Q}$.}
 $$
 Thus, keeping in mind  Theorem \ref{knapp_stein}, taking $d(I)=ak$,
 \begin{equation*}\begin{split}
 \|\alpha \|_{W^{ak,2}(\he n,E_0^h)} & \le C\big(  \| X^I\alpha \|_{L^2 (\he n,E_0^h)} +  \| \alpha \|_{L^2 (\he n,E_0^h)}\big) 
 \\&=
  C\big(  \| O_0(X^I K_k) \Delta_{\he{},h}^k \alpha \|_{L^2 (\he n,E_0^h)} +  \| \alpha \|_{L^2 (\he n,E_0^h)}\big) 
 \\& \le C \big(  \|  \Delta_{\he{},h}^k \alpha \|_{L^2 (\he n,E_0^h)} +  \| \alpha \|_{L^2 (\he n,E_0^h)}\big) .
 \end{split}\end{equation*}
Then the assertion follows by density, thanks to Lemma \ref{density S0}.
 \end{proof}

\subsection{Heat equation on $E_0^\bullet$}
We consider now the heat operator associated with the Rumin's Laplacian  $\Delta_{\he{},h} $, i.e.
$$
\mathcal L:=  \partial_s +  \Delta_{\he{},h} \qquad\mbox{in ${\mathbb R_+}\times \he n$,}
$$
where $\partial_s$ {stands for} $\partial_s  I_d$,  $I_d$ being the identity $N_h\times N_h$ matrix.
 Arguing as in \eqref{laplace matrix}, $\mathcal L$ can be written as a matrix-valued operator
 in the form
 \begin{equation}\label{heat coordinates}
\big( \delta_{i,j}  \partial_s +  \Delta_{\he{},h}^{i,j}\big)_{i,j=1,\dots N_h}=: \big(\mc L^{i,j}\big)_{i,j=1,\dots N_h},
\end{equation}
where $\delta_{i,j}$ is the Kronecker symbol, so that, arguing as in \eqref{laplace matrix},
if $T_1,\dots,T_{N_h}\in \mc D'({\mathbb R_+}\times \he n)$, with the convention of Remark \ref{current simply},
\begin{equation}\begin{split}\label{heat coordinates applied}
\mathcal L & (\sum_j T_j(\xi_j)^*) = \sum_i \big(\sum_j \mc L^{i,j}T_j\big)(\xi_i)^*
\\&
=\sum_i \big(\partial_s T_i + (\Delta_{\he{},h} T)_i \big) (\xi_i)^*. 
\end{split}\end{equation}

The following results are basically contained in \cite{folland_stein}, Chapter 4.B and \cite{ponge_mams} (in particular Lemma 5.4.9).
However, we point out that the arguments of \cite{folland}, \cite{folland_stein}
rely on the fact that the heat kernel is nonnegative (Hunt's theorem).
Clearly this is not {the case} in the present situation, since $h$ is
a vector-valued kernel.

Arguing as in \cite{folland_stein}, Chapter 4.B and keeping in mind that $\Delta_{\he{},h}$ is
a Rockland operator (Theorem \ref{global solution}, i) above), we have:
\begin{theorem}\label{hypoellipticity FS}
The operator $\mc L$ is hypoelliptic { on} ${\mathbb R_+}\times \he n$.
\end{theorem}

\begin{proposition}\label{strongly prop} 
{If $0\le h \le 2n+1$, }the operators $-\Delta_{\he{},h}:\mc D(\Delta_{\he{},h}) \subset L^2(\he n, E_0^h)\to L^2(\he n, E_0^h)$ are densely defined,
self-adjoint and dissipative, and therefore
generate strongly continuous analytic semigroup $\big( \exp(-s\Delta_{\he{}}) \big)_{s\ge 0}$ in $L^2(\he n, E_0^h)$.

In addition, 
\begin{enumerate}[i)]
\item for any $s>0$ the operator $\exp(-s\Delta_{\he{},h}) $ is left-invariant;
\item  if $I\subset [0,\infty)$ is a compact interval then
\begin{equation}\label{strongly}
\sup_{s\in I} \|  \exp(-s\Delta_{\he{}, h})\|_{\mc L (L^2(\he n, E_0^h),L^2(\he n, E_0^h))} = C_I<\infty;
\end{equation}
\item { for any $s>0$,} if { $\alpha\in\mc D(\he n, E_0^h)$}, then $ \exp(-s\Delta_{\he{}, h})\alpha \in\mc E(\he n, E_0^h)$;
\item { for any $s>0$,
\begin{equation}\label{feb 9 eq:1}
 \exp(-s\Delta_{\he{}, h}) : \mc D(\he n, E_0^h) \to  \mc D'(\he n, E_0^h).
 \end{equation}
 }

\end{enumerate}
\end{proposition}

\begin{proof}
By a density argument, $\Delta_{\he n,h}$ is symmetric since
 is formally self-adjoint in $\mc D(\he n, E_0^h)$.
In addition, arguing as in \cite{FT4}, Proposition 6.18, $\Delta_{\he{},h}$ is self-adjoint and dissipative, so that generates an
analytic semigroup $\big( \exp(-s\Delta_{\he{},h}) \big)_{s\ge 0}$ (see, e.g. \cite{kato}, Example 1.25). 
  Thus, by \cite{lunardi}, Proposition 2.1.4, $ \exp(-s\Delta_{\he{},h})$ is strongly continuous on $[0,\infty)$.
  
 Assertion i) follows straightforwardly by the left-invariance of $\Delta_{\he{},h}$,
whereas assertion ii) follows  by Banach-Steinhaus' Theorem. As for iii), take now
$k>Q/2a$ and $\beta\in (0, \min\{1,ak-Q/2\})$, so that, by \cite{folland}, Theorem 5.15
and Proposition 5.10,
$$
W^{ak,2}(\he n) \hookrightarrow \Gamma_\beta(\he n),
$$
{where the Folland-Stein H\"older spaces $\Gamma_\beta(\he n)$ will be defined in Section \ref{Folland-Stein-Sobolev}. }
If {$s\in I$},
\begin{equation}\label{Feb 7 eq:1 bis}\begin{split}
 \| \exp(-s\Delta_{\he{}}) & \alpha\|_{\Gamma_\beta(\he n)}
\le C \| \exp(-s\Delta_{\he{}})\alpha\|_{W^{ak,2}(\he n)}  
\\&
\le
C\Big\{\|\Delta_{\he{}}^k \exp(-s\Delta_{\he{}})\alpha\|_{L^2(\he n)} 
+ \| \exp(-s\Delta_{\he{}})\alpha\|_{L^2(\he n)}\Big\},
\end{split}\end{equation}
by Proposition \ref{equivalent norm}.
Let us consider the first term; the second one can be handled
in the same way. By  \cite{lunardi}, Proposition 2.1.1 and \eqref{strongly} above
\begin{equation*}\begin{split}
\|\Delta_{\he{}}^k & \exp(-s\Delta_{\he{}})\alpha\|_{L^2(\he n)} 
= \| \exp(-s\Delta_{\he{}})\Delta_{\he{}}^k \alpha\|_{L^2(\he n)} 
\\&
\le C_I \|\Delta_{\he{}}^k \alpha\|_{L^2(\he n)} \le C_I \| \alpha\|_{W^{ak,2}(\he n)} <\infty .
\end{split}\end{equation*}

Thus iii) is proved. Finally, iv) follows trivially from iii).
\end{proof}

\begin{remark}\label{t continuity} Suppose 
{ $\alpha\in\mc D(\he n, E_0^h)$}. The arguments of the proof of iii) in Proposition \ref{strongly prop}
(with the same notations)
yield also that, if $s,s' \ge 0$ then 
\begin{equation*}\begin{split}
\big| & \Scal{ h(s,\cdot)}{  \alpha} - \Scal{ h(s',\cdot)}{  \alpha} \big| 
\\&
{\le 
C \| \exp(-s\Delta_{\he{}})  \ccheck\alpha- \exp(-s'\Delta_{\he{}}) \ccheck \alpha\|_{\Gamma_\beta(\he n,E_0^h)} 
}
 \\&
{
\le C\Big\{  \|   \exp(-s\Delta_{\he{}})  \Delta_{\he,h}^k \ccheck\alpha- \exp(-s'\Delta_{\he{}}) \Delta_{\he,h}^k  \ccheck\alpha\|_{L^2(\he n, E_0^h)} 
}
\\&
{ \hphantom{C\Big\{\; } + \|   \exp(-s\Delta_{\he{}}) \ccheck \alpha- \exp(-s'\Delta_{\he{}})  \ccheck\alpha\|_{L^2(\he n, E_0^h)} \Big\} \rightarrow 0 \qquad\mbox{as $s'\to s$, }
}
%\\&
%\le C \| \exp(-s\Delta_{\he{}})  \alpha- \exp(-s'\Delta_{\he{}})  \alpha\|_{\Gamma_\beta(\he n,E_0^h)} 
% \\&
%\le C\Big\{  \|   \exp(-s\Delta_{\he{}})  \Delta_{\he,h}^k \alpha- \exp(-s'\Delta_{\he{}}) \Delta_{\he,h}^k  \alpha\|_{L^2(\he n, E_0^h)} 
%\\&
%\hphantom{C\Big\{\; } + \|   \exp(-s\Delta_{\he{}})  \alpha- \exp(-s'\Delta_{\he{}})  \alpha\|_{L^2(\he n, E_0^h)} \Big\} \rightarrow 0 \qquad\mbox{as $s'\to s$, }
\end{split}\end{equation*}
since the semigroup is strongly continuous. This proves that 
\begin{equation}\label{may 10 eq:2}
\mbox{the map }
s\to \Scal{ h(s,\cdot)}{ \alpha} \mbox{ is continuous.}
\end{equation}

In particular, if $I\subset [0,\infty)$ is a compact interval, then
\begin{equation}\label{boundedness}
\sup_{s\in I} \big| \Scal{ h(s,\cdot)}{ \alpha}\big| <\infty.
\end{equation}
%In addition, if $k>Q/2a$,
%\begin{equation}\label{boundedness quantitative}
%\sup_{s\in I} \big| \Scal{ h(s,\cdot)}{ \alpha}\big| \le C\|\alpha\|_{W^{ak,2}(\he n)}.
%\end{equation}
{
In addition, if $k>Q/2a$,
\begin{equation}\label{boundedness quantitative bis}
\sup_{s\in I} \big| \Scal{ h(s,\cdot)}{ \alpha}\big| \le C\|\ccheck \alpha\|_{W^{ak,2}(\he n)}.
\end{equation}
}

\end{remark}

%\begin{definition}
%If $i,j=,\dots,N_h$ {and $s>0$}, by Schwartz kernel theorem (see \cite{treves}, Chapter 51) 
%there exists $h_{i,j} \in \mathcal S'(\he n)$ such that
%$$\scal{\exp(-s\Delta_{\he{},h})(u\xi_j)}{\xi_i} = u \ast h_{i,j}$$
%for all $u\in\mc S(\he n)$.\footnote{Forse questa definizione andrebbe inserita nella dimostrazione della proposizione seguente.
%}
%\end{definition}

\begin{proposition} For any $s>0$, { by Proposition \ref{strongly prop},
i) and iv),} there exists a  matrix-valued kernel 
\begin{equation}\label{kernel existence}
h= h(s,p) =(h_{ij}(s,p))_{i,j=1,\dots, N_h} \in { \big(\mathcal D'(\he n)\big)^{N_h\times N_h}}
\end{equation}
such that  
$$
\exp(-s\Delta_{\he{},h})\alpha= \alpha\ast h(s,\cdot) \qquad\mbox{for  { $\alpha\in\mc D(\he n, E_0^h)$.}}
$$
Here, if $\alpha = \sum_j\alpha_j \xi_j$, then $\alpha\ast h := \sum_i \big( \sum_j \alpha_j\ast h_{ij}\big) \xi_i$.

In addition  
\begin{itemize}
\item[i)] If $s>0$
\begin{equation}\label{jan11 eq:1 bis}\begin{split}
\Scal{ h(s,\cdot)}{ \alpha}
 = \Big(\exp(-s\Delta_{\he{},h}){\color{brown}(\ccheck\alpha)}\Big)(e),
\end{split}\end{equation}

\item[ii)] If $s>0$
\begin{equation}\label{semigroup derivative bis}
\Scal{ h(s,\cdot)}{ \ccheck \Delta_{\he{},h}\alpha}=-\frac{\partial}{\partial s} \Scal{ h(s,\cdot)}{ \ccheck\alpha}.
\end{equation}
\end{itemize}
\end{proposition}

\begin{proof}
Since the convolution maps { $\mc D\times \mc D'$} into $\mc E$
{ (see \cite{treves}, Theorem 27.3)}, keeping 
 in mind  Proposition \ref{strongly prop} { and \eqref{check conv}, for all $\alpha\in\mc D(\he n, E_0^h)$,} we have:
{
\begin{equation}\label{jan11 eq:1 bis}\begin{split}
\Scal{ h(s,\cdot)}{ \alpha}
&:= \sum_{i,j}  \Scal{h_{ij}(s,\cdot)}{ \alpha_j} \xi_i 
=\sum_{i,j}  \Scal{\ccheck h_{ij}(s,\cdot)}{ \ccheck\alpha_j} \xi_i = \lim_{p\to e}  \ccheck \alpha \ast h(s,\cdot) (p)
\\&
=  \lim_{p\to e} \Big(\exp(-s\Delta_{\he{},h}){(\ccheck\alpha)}\Big) (p)
 = \Big(\exp(-s\Delta_{\he{},h}){(\ccheck\alpha)}\Big)(e).
\end{split}\end{equation}
}
This proves i). On the other hand, since both $\exp(-s\Delta_{\he{},h})\ccheck\alpha$ and
$\exp(-s\Delta_{\he{},h})\ccheck \Delta_{\he{},h}\alpha$ are smooth functions, it follows
from the identity
$$
-\frac{\partial}{\partial s}\exp(-s\Delta_{\he{},h}) = \exp(-s\Delta_{\he{},h}) \Delta_{\he{},h}
$$
that the same identity holds  at $e$. Then 
ii) follows.

\end{proof}

{
\begin{proposition}\label{h symmetric}
For $i,j=1,\dots, N_h$ we have $\ccheck h_{i,j} = h_{j,i}$, i.e.
$$
\ccheck h = {\vphantom i}^{\mathrm t}\!\,h.
$$
\end{proposition}

\begin{proof} By the spectral theorem,  $\exp(-s\Delta_{\he{},h}$) is self-adjoint in $L^2(\he n, E_0^h)$
for $s>0$. 
Thus, if $\phi,\psi$ are arbitrary test functions, then, by \eqref{convolutions var bis},
\begin{equation*}\begin{split}
\int ( &\psi\ast \ccheck h_{i,j}) \phi  \, dp
=\int (\phi \ast  h_{i,j}) \psi \, dp =
\int \scal{ \big(\exp(-s \Delta_{\he{},h}\big)(\phi\xi_j)}{\psi\xi_i)}\, dp
\\&
=
\int\scal{{\exp(-s \Delta_{\he{},h}\big)(\psi\xi_i)}}{\phi\xi_j)}\, dp
\\&
=
\int (\psi \ast  h_{j,i}) \phi \, dp.
\end{split}\end{equation*}
Thus
$$
\psi\ast\ccheck h_{i,j}  = \psi \ast  h_{j,i}.
$$
Take now $\psi=\psi_k$, where $(\psi_k)_{k\in \mathbb N}$ is a sequence
in $\mc D(\he n)$ {supported in a fixed neighborhood of $e$ and convergent in $\mc E'$ to the Dirac $\delta$ concentrated at $p=e$ (see \cite{treves}, Theorem 28.2)}. 
Taking the limit as $k\to\infty$, by \cite{schwartz}, Th\'eor\`eme
V, p.157,
$\ccheck h_{i,j} =\delta\ast\ccheck h_{i,j}   = \delta \ast  h_{j,i} = h_{j,i}$,
and the assertion follows.
\end{proof}
}

\begin{definition} The kernel $h=h(s,p)$ can be identified with a matrix-valued distribution
$$
h\in  \big(\mc D'((0,\infty)\times\he n)\big)^{N_h\times N_h}
$$ 
as follows: first of all we notice that
by \cite{treves}, Theorem 39.2, a distribution in $(0,\infty)\times \he n$ can be defined
by its action on $\mc D(\mathbb R)\otimes \mc D(\he n)$. Thus, arguing on the
entries of $h$, if
$v\in \mc D((0,\infty))$ and $\alpha\in\mc D(\he n, E_0^h)$, we can set
\begin{equation}\label{july12 eq:1}
\Scal{h}{v\otimes \alpha}=
\int_I  v(s) \Scal{h(s,\cdot)}{\alpha}\, ds
:=  \sum_{i,j} \big(\int_I  v(s)  \Scal{h_{ij}(s,\cdot)}{ \alpha_j} \, ds \big)\,\xi_i.
\end{equation}
Keeping in mind {\eqref{boundedness}, \eqref{july12 eq:1}}
defines a distribution.

\end{definition}

\begin{proposition}\label{Lh}
We have:
\begin{enumerate}[i)]
\item $\mc L h = 0$ in $\big(\mc D'(\mathbb R_+\times \he n)\big)^{N_h\times N_h}$,
i.e.
\begin{equation}\label{Aug 10 eq:1}
\Scal{\mc Lh}{A(s,y)}=0 \qquad\mbox{for all $A\in\mc D((0,\infty)\times \he n$,}
\end{equation}
where the action of the heat operator $\mc L$ on $h$ 
{must be understood} as
the formal matrix product
\begin{equation}\label{Lh def}
\mc L h: =\big(\mc L^{i,j}\big)_{i,j=1,\dots,N_h} \cdot\big( h_{i,j}\big)_{i,j=1,\dots,N_h},
\end{equation}
{defined in the sense of distributions.}

\item the matrix-valued distribution  $h$ is smooth on 
$(0,\infty)\times \he n$.
%by Theorem \ref{hypoellipticity FS}.
In particular, if $\phi=\sum_j \phi_{j} \in \mc D(\mathbb R_+\times \he n, E_0^h)$, we can write
\begin{equation}\label{h smooth}
\Scal{h}{\phi} =\sum_{i,j}  \big( \int_{\mathbb R_+\times \he n} h_{i,j}(s,p)\phi_{j}(s,p)\, ds\, dp\big) \xi_i;
\end{equation}
\item if $r>0$ 
\begin{equation}\label{jan11 eq:3}
h (r^as,y) = r^{-Q} h(s, \delta_{1/r}y)\qquad\mbox{for $s>0$ and $y\in \he n$;}
\end{equation}
\item[iv)] combining ii) and i), it follows that
\begin{equation}\label{Aug11 eq:2}
\partial_s h_{i,j} - \sum_\ell (\Delta_{\he{},h})^{i,\ell}h_{\ell,j} =0,\qquad i,j=1,\dots,N_h.
\end{equation}

\end{enumerate}

\end{proposition}

\begin{proof}
%Assertion i) follows straightforwardly, again by checking the identity on forms
%in $\mc D(\mathbb R)\otimes \mc D(\he n, E_0^h)$ and keeping in mind \eqref{semigroup derivative}.

%QUI
 To prove assertion i), by \cite{treves}, Theorem 39.2, we check the identity on forms
$$
\psi\otimes A\in\mc D(\mathbb R)\otimes \mc D(\he n, (E_0^h)^* \otimes E_0^{h }).
$$
We have:
\begin{equation*}\begin{split}
&\Scal{ (\partial_s +  \Delta_{\he{},h})h}{\psi\otimes A} : = \Scal{ h}{(-\partial_s +  \Delta_{\he{},h})(\psi\otimes A)}
\\&\hphantom{xxxx}
{=\Scal{h}{ -\psi'(s) \otimes A + \psi(s)\otimes  \Delta_{\he{},h} A}}
\\&\hphantom{xxxx}
= - \sum_{i,j}\int \, ds\,  \psi'(s) \Scal{h_{ij}}{A_{i,j}} + \sum_{i,j} \int \, ds\,  \psi(s) \Scal{h_{i,j}}{( \Delta_{\he{},h}A)_{i,j}}.
\end{split}\end{equation*}

Put 
$$
\Phi_j = \sum_\ell A_{\ell,j} \xi_\ell, \qquad \mbox{{so that $A_{\ell,j} = (\Phi_j)_\ell$}}.
$$

Now we have
\begin{equation*}\begin{split}
\sum_{i,j} & \Scal{h_{i,j}}{( \Delta_{\he{},h}A)_{i,j}} = \sum_{i,j}\sum_{\ell} \Scal{h_{i,j}}{\Delta_{\he{},h}^{i,\ell}A_{\ell,j}}
\\&\hphantom{xxxx}
=\sum_{i,j}\sum_{\ell} \Scal{h_{i,j}}{\Delta_{\he{},h}^{i,\ell}{(\Phi_j)_\ell}}
=\sum_{i,j} \Scal{h_{i,j}}{ (\Delta_{\he{},h}\Phi_j)_i}
\\&\hphantom{xxxx}
=\sum_{i,j} \Scal{h_{j,i}}{ \ccheck( \Delta_{\he{},h}\Phi_j)_i}
= \sum_j \big(\exp(-s \Delta_{\he{},h})\Delta_{\he{},h}\Phi_j(e))_j
\\&\hphantom{xxxx}
= -\dfrac{\partial}{\partial s} \sum_j \big(\exp(-s \Delta_{\he{},h})\Phi_j(e))_j.
\end{split}\end{equation*}
Integrating by parts, 
\begin{equation*}\begin{split}
\sum_{i,j} \int \, ds\,  \psi(s) \Scal{h_{i,j}}{( \Delta_{\he{},h}A)_{i,j}}
=  \sum_j \int \, ds\,  \psi'(s)  \big(\exp(-s \Delta_{\he{},h})\Phi_j(e))_j
\end{split}\end{equation*}
On the other hand,
\begin{equation*}\begin{split}
\sum_{i,j} & \Scal{h_{i,j}}{A_{i,j}}  = 
{\sum_{i,j} \big(\Phi_j)_i\ast h_{j,i} (e)}
= \sum_{j} \big(\exp(-s  \Delta_{\he{},h})\Phi_j (e)\big)_j
\end{split}
\end{equation*}
and the assertion is proved.

To prove ii), let us consider the  currents $H_\lambda=\sum_\ell h_{\ell,\lambda}\xi_\ell^*$, for $\lambda=1,\dots, N_h$.
We want to show that 
$$
(\partial_s +  \Delta_{\he{},h}) H_\lambda =0.
$$
If $\alpha= \sum\alpha_\ell\xi_\ell$ is a test form and $\psi\in\mc D(\mathbb R)$, this means that
$$
\Scal{(\partial_s +  \Delta_{\he{},h}) H_\lambda}{\psi\otimes\alpha}=0.
$$
%Let now $(\phi_N)_{N\in\mathbb N}$ be a sequence of smooth cut-off functions
%of $B(e,N)$ and $B(e,N+1)$. Obviously, $\phi_N\alpha \to \alpha$ in $\mc D(\he n,E_0^h)$
%and then it will be enough to show that
%$$
%\Scal{(\partial_s +  \Delta_{\he{},h}) H_\lambda}{\psi\otimes\phi\alpha}=0\qquad\mbox{for all $\phi\in\mc D(\he n)$.}
%$$
Now
\begin{equation*}\begin{split}
&\Scal{(\partial_s +  \Delta_{\he{},h}) H_\lambda}{\sum_\ell (\psi\otimes\alpha_\ell)\xi_\ell}
\\&\hphantom{xxx}
= \Scal{ H_\lambda}{\sum_{j,\ell} \big(-\partial_s + \Delta_{\he{},h}^{j\ell}\big)(\psi\otimes\alpha_\ell)\xi_j}
= \sum_{i,\ell} \Scal{ h_{i,\lambda}}{\big(-\partial_s + \Delta_{\he{},h}^{i\ell}\big)(\psi\otimes\alpha_\ell)}
\\&\hphantom{xxx} = \sum_{i,j, \ell} \Scal{ h_{i,j}}{\big(-\partial_s + \Delta_{\he{},h}^{i\ell}\big)(\psi\otimes \delta_{\lambda,j}\alpha_\ell)} = 0
\end{split}\end{equation*}
by i), if we choose {$A^{(\lambda)}_{\ell,j} :=  \delta_{\lambda,j}\alpha_\ell$}.

Thus, by Theorem \ref{hypoellipticity FS}, $H_\lambda$ is smooth for $\lambda=1,\dots, N_h$,
and hence the $h_{i,j}$'s are smooth for $i,j=1,\dots, N_h$.

Finally, to prove iii) let us consider the case $a=2$. The case $a=4$
can be handled in the same way. Keeping in mind \eqref{h smooth} and \cite{treves}, 
Theorem 39.2, by density it will be enough to prove that, if $u=\sum_j u_j \xi_j\in \mc D(\he n, E_0^h)$
and $v\in\mc D(0,\infty)$, then

\begin{equation}\label{to prove 2}\begin{split}
\Scal{h(r^2s,y)}{v\otimes u}&
=\sum_{i,j} \big\{ \int_0^\infty \Big(\int_{\he n } h_{i,j} (r^2s,y)u_j(y)\,dy \Big) v(s)\, ds \big\}\xi_i
\\& =
r^{-Q} \sum_{i,j}\big\{ \int_0^\infty \Big(\int_{\he n } h_{i,j}(s, \delta_{1/r}y)u_j(y)\,dy \Big) v(s)\, ds \big\} \xi_i
\\&
= r^{-Q} \Scal{h(s,\delta_{1/r}y) }{v\otimes u}.
\end{split}\end{equation}

Now, if we put $v_r(s):= v(r^{-2}s)$,
\begin{equation*}\begin{split}
\sum_{i,j} \big\{\int_0^\infty &\Big(\int_{\he n } h_{i,j} (r^2s,y)u_j(y)\,dy \Big) v(s)\, ds\big\}  \xi_i
\\&=
r^{-2} \sum_{i,j} \big\{ \int_0^\infty \Big(\int_{\he n } h_{i,j} (s,y)u(y)\,dy \Big) v(r^{-2}s)\, ds \big\}\xi_i
\\&=:
r^{-2} \Scal{h}{v_r\otimes u}
\\&=
r^{-2}\int_0^\infty v_r(s)\big( \exp(-s\Delta_{\he{}}) \ccheck u\big)(e)\,ds
\\&=
{\int_0^\infty v(s)\big( \exp(-r^2s\Delta_{\he{}})\ccheck u\big)(e)\,ds}
\\&=
\int_0^\infty v(s)\big( \exp(-s\Delta_{\he{}})(\ccheck u\circ\delta_r)\big)(e)\,ds
\\&
\qquad\mbox{(since $ \exp(-s\Delta_{\he{}})$ commutes with group dilations)}
\\& =
\Scal{h}{v\otimes u\circ \delta_r} 
\\&=
\sum_{i,j} \big\{ \int_0^\infty \Big(\int_{\he n } h (s,y)u_j(\delta_ry)\,dy \Big) v(s)\, ds\big\}\xi_i
 \\&=
r^{-Q} \sum_{i,j} \big\{ \int_0^\infty \Big(\int_{\he n } h (s,\delta_{1/r}y))u_j(y)\,dy \Big) v(s)\, ds\big\}\xi_i, 
\end{split}\end{equation*}
and \eqref{to prove 2} follows.

\end{proof}

\begin{theorem}\label{heat kernel} Denote by $\tilde h$ the matrix-valued function on $\mathbb R\times \he n$
 defined continuing $h$ by zero for $s\le 0$.
 Then \begin{enumerate}[i)]
 \item keeping in mind \eqref{h smooth},    $\tilde h$ defines a matrix-valued distribution $$\tilde h\in \big(\mc D'(\mathbb R\times\he n)\big)^{N_h\times N_h}$$
 by the identity
\begin{equation}\label{tilde h}
\Scal{\tilde h}{v\otimes u}:= \int_0^\infty v(s) \Scal{h(s,\cdot)}{u} \,ds
\end{equation}
 when $v\in \mc D(-\infty,\infty)$ and $u\in\mc D(\he n,E_0^h)$;
 \item $\mc L\tilde h =\delta_0\otimes\delta_{e, h}$, where $\delta_{e, h}$
 is the matrix-valued distribution $(a_{i,j})_{i,j=1,\dots,N_h}$ where
$a_{i,j} = 0$ if $i\neq j$, and $a_{i,i}=\delta_e$ for $i=1,\dots,N_h$, so that
$$
\delta_{e, h}u = u(0)\qquad\mbox{for all $u\in \mc D(\he n,E_0^h)$;}
$$

 \item $\tilde h \in \mathbb C^\infty \big((\mathbb R\times \he n) \setminus (0,e)\big)$.
 \end{enumerate}

\end{theorem}

\begin{proof} Again by \cite{treves}, Theorem 39.2, {identity} \eqref{tilde h}
defines a distribution. This proves i).

Let us prove ii). Arguing as in \eqref{Aug 10 eq:1},
\begin{equation*}\begin{split}
&\Scal{\mc L\tilde h}{v\otimes u} = \Scal{\tilde h}{-\partial_s v\otimes u 
+ v\otimes \Delta_{\he{}}u}
\\&\hphantom{xxx}
= -\int_0^\infty \partial_s v(s)\Scal{h(s,\cdot)}{u}
+ \int_0^\infty v(s) \Scal{h(s,\cdot)}{ \Delta_{\he{},h} u} \,ds.
\end{split}\end{equation*}
By \eqref{boundedness quantitative bis}, {the integrals}
$$
\int_{\he n} h(s,y) u(y)\, dy \qquad\mbox{and} \qquad \int_{\he n} h(s,y) \Delta_{\he{}}u(y)\, dy
$$
{are both bounded for $s\in \supp v\cap [0,\infty)$,
hence}
$$
\partial_sv(s) \int_{\he n} h(s,y) (y)u\, dy \qquad\mbox{and}
 \qquad v(s)\int_{\he n} h(s,y) \Delta_{\he{}}u(y)\, dy
$$
belong to $L^1([0,\infty))$ and we can write
\begin{equation*}\begin{split}
-\int_0^\infty  &\partial_s v(s) \big( \int_{\he n} h(s,y) u(y)\, dy\big)
+ \int_0^\infty v(s) \big( \int_{\he n} h(s,y) \Delta_{\he{}}u(y)\, dy\big) \,ds
\\&
= \lim_{\eps\to 0}
\Big\{ -\int_\eps^\infty  \partial_s v(s) \big( \int_{\he n} h(s,y) u(y)\, dy\big)
\\&\hphantom{xxxxx}
+  \int_\eps^\infty v(s) \big( \int_{\he n} h(s,y) \Delta_{\he{}}u(y)\, dy\big) \,ds\Big\}
:= \lim_{\eps\to 0}\{I_\eps + J_\eps\} .
\end{split}\end{equation*}
Since
$$
(s,y) \to \partial_s v(s)u(y) \qquad\mbox{{belongs to } $L^1([\eps,\infty)\times \he n)$}
$$
we have
\begin{equation*}\begin{split}
I_\eps =-\int_\eps^\infty & \partial_s v(s) \big( \int_{\he n} h(s,y) u(y)\, dy\big)
= -\int_{\he n}u(y)  \big( \int_\eps^\infty  h(s,y) \partial_s v(s)\, ds\big)\, dy
\\&
=v(\eps) \int_{\he n}u(y)   h(\eps,y)\, dy
+ \int_{\he n}u(y)  \big( \int_\eps^\infty  \partial_s h(s,y) v(s)\, ds\big)\, dy.
\end{split}\end{equation*}
%Thus, keeping in mind Proposition \ref{strongly} iii)-2)
%\begin{equation*}\begin{split}
%I_\eps + J_\eps &= v(\eps) \int_{\he n}u(y)   h(\eps,y)\, dy
%= v(\eps) \exp(-\eps\Delta_{\he{}})u(e).
%\end{split}\end{equation*}
{
Thus, keeping in mind Proposition \ref{strongly} iii)-2)
\begin{equation*}\begin{split}
I_\eps + J_\eps &= v(\eps) \int_{\he n}u(y)   h(\eps,y)\, dy
= v(\eps)\big( \exp(-\eps\Delta_{\he{}})\ccheck u\big)(e).
\end{split}\end{equation*}}

Arguing as in Remark \ref{t continuity}, if we 
take $k>Q/2a$
%\begin{equation*}\begin{split}
%\Big|\exp(-\eps\Delta_{\he{}})u(e) &-u(e)\Big| \le
%C\| (\exp(-\eps\Delta_{\he{}})-I) \Delta_{\he{}}^ku\|_{L^2(\he n, E_0^h)} \to 0
%\end{split}\end{equation*}
{
\begin{equation*}\begin{split}
\Big| \Big(\exp(-\eps\Delta_{\he{}})\ccheck u\Big) (e) &
-u(e)\Big|
%\Big|  (-\eps\Delta_{\he{}})\ccheck u\Big) (e) 
%-u(e)\Big| 
= \Big| \Big(\exp(-\eps\Delta_{\he{}})\ccheck u\Big) (e) 
-\ccheck u(e)\Big| 
\\&
\le
C\| (\exp(-\eps\Delta_{\he{}})-I) \Delta_{\he{}}^k (\ccheck u)\|_{L^2(\he n, E_0^h)} \to 0
\end{split}\end{equation*}
}
as $\eps\to 0$, since $s\to \exp(-s\Delta_{\he{}})$ is a strongly continuous
semigroup.
{This proves ii)}.
Finally, iii) follows straighforwardly from ii) and Theorem \ref{hypoellipticity FS}.

\end{proof}

\begin{theorem}\label{to molecules} For any $s,\sigma>0$ we have:
\begin{enumerate}[i)]
\item $h_{i,j}(s,\cdot)\in \mc S(\he n)$, $i,j=1,\dots, N_h$;
\item $h(s,\cdot)\ast h(\sigma, \cdot) = h(\sigma,\cdot)\ast h(s, \cdot)
= h(s+\sigma,\cdot)$.
\end{enumerate}
\end{theorem}

\begin{proof} { By the very definition of $\tilde h$, 
$k,\ell\in \mathbb N\cup\{0\}$, then the map $(s,y)\mapsto s^{-k}\, \partial_s^\ell \tilde h_{i,j}(s,y)$
is continuous away from $(0,e)$.

Thus, if  $K\subset \he n$ is a compact set, $e\notin K$,
$$
\sup_K \big| s^{-k}\, \partial_s^\ell h_{i,j}(s,\cdot)\big| \to 0\qquad\mbox{as $s\to 0$}
$$
}
{ Then the proof of i)} can be carried out as in \cite{folland_stein}, Proposition 1.74.
In addition, i) implies that the convolutions in ii) well defined and 
{so ii) follows from} the semigroup property of $s\to \exp(-s\Delta_{\he{}})$.

\end{proof}

Thanks to the density of $\mc D(\he n)$ in $\mc D'(\he n )$, the following lemma holds (see \cite{treves}, p. 272)
 \begin{lemma}\label{jan10 lemma1} If $T\in \mc D'(\he n)$ and there exist $C>0$ and $N\in \mathbb N$ such that
$$
\big| \Scal{T}{\phi} \big| \le C \sup_{m+|\alpha|\le N}\sup_{p\in\he n}(1+|p|)^m|D^\alpha\phi(p)|
$$
for all $\phi\in\mc D(\he n)$, then $T\in \mc S'(\he n)$.

\end{lemma}

Combining the previous lemma with Theorem XIII p. 74 of \cite{schwartz}
we have:
\begin{proposition} \label{jan13 prop:1} Let $(T_j)_{j\in\mathbb N}$ a sequence in 
$\mc S'(\he n)\subset \mc D'(\he n)$ such that
\begin{itemize}
\item[i)] $(T_j)_{j\in\mathbb N}$ is bounded in $\mc S'(\he n)$, i.e. there exist $C>0$ and $N\in \mathbb N$ such that
\begin{equation}\label{jan10 eq:5}
\big| \Scal{T_j}{\phi} \big| \le C \sup_{m+|\alpha|\le N}\sup_{p\in\he n}(1+|p|)^m|D^\alpha\phi(p)|
\end{equation}
for all $\phi\in\mc D(\he n)$ and $j\in\mathbb N$;
\item[ii)] the sequence $(\Scal{T_j}{\phi})_{j\in\mathbb N} $ has a limit $\Scal{T}{\phi}$ as $j\to\infty$
 for all $\phi\in \mc D(\he n)$,
\end{itemize}
then $T\in\mc S'(\he n)$ and $T_j\to T$ in $\mc D'(\he n)$ as $j\to\infty$.

\end{proposition}

\begin{proof} By Theorem XIII p. 74 of \cite{schwartz}, $T\in\mc D'(\he n)$.
On the other hand,  \eqref{jan10 eq:5} still holds for $T$, and the assertion
follows from Lemma \ref{jan10 lemma1}.

\end{proof}

\begin{proposition}\label{Aug11 prop:1} 
If {$\phi\in\mc D(\he n)$ and}
$i,j=1,\dots, N_h$ 
we have:
\begin{itemize} 
\item[i)]
{ by \eqref{may 10 eq:2}, the function
$s\to  \Scal{ h_{i,j}(s, \cdot)}{\phi}$ is continuous for $s\ge 0$,} and the identity
\begin{equation}\label{7gen eq:2}
\Scal{\int_0^M h_{i,j}(s, \cdot)\, ds}{\phi}:=\int_0^M \Scal{ h_{i,j}(s, \cdot)}{\phi}\, ds
\end{equation}
defines a tempered distribution for all $M>1$;
\item[ii)] the function $s\to  \Scal{ h_{i,j}(s, \cdot)}{\phi}$ belongs to $L^1([0,\infty))$ 
\item[iii)] the  identity
\begin{equation}\label{16jan eq:1}
\Scal{\int_0^\infty h_{i,j}(s, \cdot)\, ds}{\phi}:=\lim_{M\to\infty} \int_0^M \Scal{ h_{i,j}(s, \cdot)}{\phi}\, ds
\end{equation}
defines a tempered distribution.    Thus
$$
\int_0^\infty h(s, \cdot)\, ds := \Big( \int_0^\infty h_{i,j}(s, \cdot)\, ds \Big)_{i,j=1,\dots, N_h}
$$
is a matrix-valued tempered distribution.
Notice that, by Proposition \ref{jan13 prop:1}, we can write also
\begin{equation}\label{Aug11 eq:1}
\int_0^\infty h_{i,j}(s, \cdot)\, ds =\lim_{M\to\infty} \int_0^M  h_{i,j}(s, \cdot)\, ds\qquad\mbox{in $\mc D'(\he n)$.}
\end{equation}

\end{itemize}

\end{proposition}

%\begin{proposition} If $i,j=1,\dots, N_h$ and $M>1$ the identity
%\begin{equation}\label{7gen eq:2}
%\Scal{\int_0^M h_{i,j}(s, \cdot)\, ds}{\phi}:=\int_0^M \Scal{ h_{i,j}(s, \cdot)}{\phi}\, ds
%\qquad\mbox{for $\phi\in\mc D(\he n)$}
%\end{equation}
%defines a distribution. In addition
%there exist $C>0$ and $N\in \mathbb N$ such that
%\begin{equation}\label{jan10 eq:5}
%\big| \int_0^M\Scal{h_{i,j}(s, \cdot)}{\phi} \big|\, ds \le C \sup_{m+|\alpha|\le N}\sup_{p\in\he n}(1+|p|)^m|D^\alpha\phi(p)|
%\end{equation}
%for all $\phi\in\mc D(\he n)$.
%\end{proposition}

\begin{proof} 

{ Let us prove the following statement: there exist $C>0$ and $N>0$ independent of $s>0$ and $\phi$ such that
\begin{equation}\label{may 10 eq:1}
|\Scal{ h_{i,j}(s, \cdot)}{\phi}| \le
C\min\{1, s^{-Q/a} \}
\sup_{m+|\alpha|\le N}\sup_{p\in\he n}(1+|p|)^m|D^\alpha\phi(p)|.
\end{equation}
Then i), ii) and iii) will follow by Proposition \ref{jan13 prop:1}.

First of all, we prove that, if $M>1$,
there exist $C_M>0$ and $N\in \mathbb N$ such that,
if $0<s\le M$,
\begin{equation}\label{jan10 eq:5 bis}
\big| \Scal{h_{i,j}(s, \cdot)}{\phi} \big| \le C_M \sup_{m+|\alpha|\le N}\sup_{p\in\he n}(1+|p|)^m|D^\alpha\phi(p)|
\end{equation}
for all $\phi\in\mc D(\he n)$. 
 Indeed, by \eqref{boundedness quantitative bis}, if  $k>Q/2a$
and $I\subset [0,\infty)$ is a compact interval, then
\begin{equation}\label{Aug21 eq:5}
\sup_{s\in I} \big| \Scal{ h(s,\cdot)}{ \alpha}\big| \le C_I\,\|\ccheck\alpha\|_{W^{ak,2}(\he n)}.
\end{equation}
On the other hand, if $J$ is a multi-index with $d(J)\le ak$ there exits a family of polynomials $P_\sigma$,
$|\sigma|\le ak$,
such that for any function $u\in \mc S(\he n)$
\begin{equation*}\begin{split}
\| W^J u\|_{L^2(\he n)}^2& \le \sum_{|\sigma|\le ak}\int_{\he n} |P_\sigma(p) D^\sigma u|^2\, dp
\\&
\le C \int_{\he n} (1+|p|)^{-2m}\, dp\cdot\sum_{|\sigma|\le ak}\sup_{p\in\he n}(1+|p|)^{2m}|D^\sigma(p)|^2 
\\&
= C\sum_{|\sigma|\le ak}\sup_{p\in\he n}(1+|p|)^{2m}|D^\sigma\phi(p)|^2
\end{split}\end{equation*}
for $m$ large enough. This proves \eqref{jan10 eq:5}.

On the other hand, keeping in mind \eqref{jan11 eq:3} and Theorem \ref{to molecules}, if 
$s>1$, then
\begin{equation}\label{jan11 eq:2 bis}\begin{split}
&|\Scal{ h_{i,j}(s, \cdot)}{\phi}| \le \int_{\he n} | h_{i,j}(s, p) \phi(p) |\, dp
\\&
\hphantom{xx} = s^{-Q/a} \int_{\he n} | h_{i,j}(1, \delta_{s^{-1/a}}p) \phi(p) |\, dp
\\&
\hphantom{xx}
\le  s^{-Q/a} \sup_{p\in \he n} |h_{i,j}(1,p)| \| \phi\|_{L^1(\he n, E_0^h)}
\\&
\hphantom{xx}
\le C  s^{-Q/a} \sup_{p\in \he n} |h_{i,j}(1,p)| \sup_{p\in \he n}  (1+|p|)^{2n+2}|\phi(p)|.
\end{split}\end{equation} 
Then, combining \eqref{jan10 eq:5 bis} and \eqref{jan11 eq:2 bis}, \eqref{may 10 eq:1} follows.
}

\end{proof}

\begin{theorem}\label{jan16 th:1} We have
\begin{equation}\label{jan16 eq:4}
\int_0^\infty h(s,\cdot)\, ds 
 =  \Delta_{\he{},h}^{-1}
\end{equation}
(identity between convolution kernels).
\end{theorem}

\begin{proof} 
%The assertion follows from Theorem \ref{symmetries} provided we prove that
First of all, let us prove that
\begin{equation}\label{jan10 eq:6}
\Delta_{\he{},h} \int_0^\infty h(s,\cdot)\, ds =  \delta_{e,h}.
\end{equation}
%and
%\begin{equation}\label{jan11 eq:5}
%\lim_{p\to\infty} \int_0^\infty h(s,p)\,ds = 0.
%\end{equation}
To this end, let  $\phi\in\mc D(\he n,E_0^\bullet))$ be a test form.
{Suppose that
$\supp\phi\subset K$, where $K\subset\he n$ is a compact set.}
We notice first that
\begin{equation}\label{jan16 eq:3}
\int_0^1 
|\Scal{  h(s,\cdot)}{\Delta_{\he{},h} \phi}| \, ds <\infty
\end{equation}
(the integral is well defined by Remark \ref{t continuity}). Then
by  \eqref{Aug11 eq:1},  Proposition \ref{Aug11 prop:1}, ii), 
 \eqref{Aug11 eq:2},
  and \eqref{jan16 eq:3}, we have:
\begin{equation*}\begin{split}
& \Scal{\Delta_{\he{},h} \int_0^\infty h(s,\cdot)\, ds}{\phi}
= \lim_{M\to\infty} \int_0^M 
\Scal{  h(s,\cdot)}{\Delta_{\he{},h} \phi}\, ds
\\&
 = \lim_{M\to\infty} \int_1^M 
\Scal{  h(s,\cdot)}{\Delta_{\he{},h} \phi}\, ds
+
\lim_{\eps\to 0}
\int_\eps^1 
\Scal{  h(s,\cdot)}{\Delta_{\he{},h} \phi}\, ds
\\&=
-\lim_{M\to\infty} \int_1^M 
\Scal{  \partial_sh(s,\cdot)}{ \phi}\, ds
-
\lim_{\eps\to 0}
\int_\eps^1 
\Scal{  \partial_s h(s,\cdot)}{ \phi}\, ds
\\&
\\&
=
-\lim_{M\to\infty} \Scal{  h(M,\cdot)}{ \phi} + \lim_{\eps\to 0} \Scal{  h(\eps,\cdot)}{ \phi}
\\&
=
  \lim_{\eps\to 0} \big(\exp(- \eps\Delta_{\he{},h})\phi\big)(e) =  \phi(e).
\end{split}\end{equation*}
{ Thus \eqref{jan10 eq:6} holds, and then 
$$
\int_0^\infty h(s,\cdot)\, ds \in \big( \mc E(\he n\setminus\{e\}, E_0^\bullet)\big)^{N_h\times N_h}
$$
since $\Delta_{\he{},h}$ is hypoelliptic, by Theorem \ref{global solution}, i). Thus, keeping
into account Proposition \ref{Aug11 prop:1}, ii), we have
}
$$
 \int_0^\infty h(s,\cdot)\, ds \in \big( \mc S'(\he n, E_0^\bullet) \cap \mc E(\he n\setminus\{e\}, E_0^\bullet)\big)^{N_h\times N_h}.
$$
and \eqref{jan16 eq:4} follows from Proposition \ref{symmetries},
 provided we prove that
\begin{equation}\label{jan11 eq:5}
\lim_{p\to\infty} \int_0^\infty h(s,p)\,ds = 0.
\end{equation}
But, thanks to \eqref{jan11 eq:3}, it follows easily that
$$
\int_0^\infty h(s,\cdot)\, ds\qquad\mbox{is a (vector-valued) kernel of type $\alpha$}
$$
(see Definition \ref{folland kernels} below), that vanishes at infinity since $Q>a$. This completes the proof of the theorem.

\end{proof}

%{\color{green}\begin{corollary}[Likely wrong, to be trashed] \label{jan11 coroll:1}
%We have:
%$$
%\ccheck\Delta_{\he{},h}\ccheck\int_0^\infty h(s,\cdot)\,ds = \delta_{e,h}
%$$ 
%\end{corollary}
%
%
%\begin{proof} We have
%\begin{equation*}\begin{split}
%\ccheck\Delta_{\he{},h} & \ccheck\int_0^\infty h(s,\cdot)\,ds = 
%\ccheck\Delta_{\he{},h}\int_0^\infty \ccheck h(s,\cdot)\,ds
%\\&=
%\ccheck\Delta_{\he{},h}\int_0^\infty  h(s,\cdot)\,ds = \ccheck \delta_{e,h} = \delta_{e,h}.
%\end{split}\end{equation*}
%
%\end{proof}
%}

\begin{corollary}\label{Aug12 cor:1}
By Theorems \ref{jan16 th:1} and \ref{symmetries}
\begin{equation*}
\Delta_{\he{},h}\int_0^\infty \ccheck h(s,\cdot)\, ds 
= \Delta_{\he{},h} \ccheck\Delta_{\he{},h}^{-1} = \Delta_{\he{},h}\Delta_{\he{},h}^{-1} = \delta_{e,h}.
\end{equation*}

\end{corollary}

\begin{lemma} If $\alpha\in L^1(\he n, E_0^\bullet) \subset \mc S' (\he n, E_0^\bullet)$ and $s>0$ then
\begin{equation}\label{Aug11 eq:3}\begin{split}
F(s,& \cdot):= h(\frac s2,\cdot) \ast d_c^*\alpha
=
 d_c^*\big( h(\frac s2,\cdot) \ast \alpha\big)
 \\&
  \in \mc O_M(\he n, E_0^\bullet)
 \subset \mc E(\he n, E_0^\bullet)\cap \mc S'(\he n, E_0^\bullet)
\end{split}\end{equation}
(we recall that $\mc O_M$ denotes the space of the smooth functions slowly
increasing at infinity: see \cite{treves}, Theorem 25.5, p.275 or \cite{schwartz},
p.243).

\end{lemma} 
\begin{proof}
Since $L^1 (\he n)\subset \mc S'(\he n)$ then both 
$h(\frac s2,\cdot) \ast d_c^*\alpha$ and $ d_c^*\big( h(\frac s2,\cdot) \ast \alpha\big)$
belong to $ \mc O_M$ (see \cite{schwartz}, p. 248). On the other hand, {given $\phi
\in \mc D(\he n, E_0^\bullet)$, we have}
\begin{equation*}\begin{split}
& \Scal{d_c^*\big( h(\frac s2,\cdot) \ast \alpha\big)}{\phi}
= \Scal{h(\frac s2,\cdot) \ast \alpha}{d_c\phi}
\\&=
\Scal{\alpha}{\ccheck h(\frac s2,\cdot) \ast d_c\phi}
=
\Scal{\alpha}{d_c\big( \ccheck h(\frac s2,\cdot) \ast \phi \big) }
\\&=
\Scal{d_c^*\alpha}{\ccheck h(\frac s2,\cdot) \ast \phi}
=
\Scal{h(\frac s2,\cdot) \ast d_c^*\alpha}{ \phi}.
\end{split}\end{equation*}

\end{proof} 

\begin{remark} Again by \cite{schwartz}, p. 248, for any $s>0$
$$
h(\frac s2,\cdot) \ast F(s,\cdot) \in \mc O_M  \subset \mc E(\he n, E_0^\bullet)\cap \mc S'(\he n, E_0^\bullet).
$$

\end{remark}

\begin{lemma}\label{jan17 lemma:1}

The function
$$
s\to  \Scal{d_c\big(h(\frac s2,\cdot) \ast F(s,\cdot)\big)}{\phi}
$$
belongs to $L^1([0,\infty))$ for all $\phi\in \mc D(\he n, E_0^\bullet)$.
In particular, for all $\phi\in \mc D(\he n, E_0^\bullet)$, there exists
$$
\lim_{M\to\infty} \int_0^M  \Scal{d_c\big(h(\frac s2,\cdot) \ast F(s,\cdot)\big)}{\phi}\, ds.
$$
\end{lemma} 

\begin{proof} If $s>0$, keeping in mind Theorem \ref{to molecules}, we have:
\begin{equation*}\begin{split}
\big| & \Scal{d_c\big(h(\frac s2,\cdot) \ast F(s,\cdot)\big)}{\phi} \big|
= \big|  \Scal{h(\frac s2,\cdot) \ast F(s,\cdot)}{d_c^*\phi} \big|
\\& =
\big|  \Scal{F(s,\cdot)}{\ccheck h(\frac s2,\cdot) \ast d_c^*\phi} \big|
=
\big|  \Scal{h(\frac s2,\cdot) \ast d_c^*\alpha}{\ccheck h(\frac s2,\cdot) \ast d_c^*\phi} \big|
\\&=
\big|  \Scal{d_c^*\alpha}{\ccheck h(s,\cdot) \ast d_c^*\phi} \big|\qquad\mbox{(by Theorem \ref{to molecules}
and \eqref{check conv})  }
\\&=
\big|  \Scal{\alpha}{\ccheck h(s,\cdot) \ast d_cd_c^*\phi} \big|
\\&\le
\int \big|  \scal{\alpha}{\ccheck h(s,\cdot) \ast d_cd_c^*\phi} \big|\, dp
\\&
\le \|\alpha\|_{L^1(\he n, E_0^\bullet)}
\cdot \|\ccheck d_cd_c^*\phi\ast h(s,\cdot)\|_{L^\infty(\he n, E_0^\bullet)}.
\end{split}\end{equation*}
On the other hand,  by \eqref{boundedness quantitative bis}
$$
\sup_{[0,1]}\|\ccheck d_cd_c^*\phi\ast h(s,\cdot)\|_{L^\infty(\he n, E_0^\bullet)} \le C_\phi,
$$
whereas, if $s>1$, by \eqref{jan11 eq:3},
\begin{equation*}\begin{split}
\|\ccheck d_cd_c^*\phi & \ast h(s,\cdot)\|_{L^\infty(\he n, E_0^\bullet)}
\le
C_\phi \| h(s,\cdot)\|_{L^\infty(\he n, E_0^\bullet)}
\\&
\le C_\phi s^{-Q/a} \| h(1,\cdot)\|_{L^\infty(\he n, E_0^\bullet)}.
\end{split}\end{equation*}
Since {\color{brown}$a<Q$} the assertion is proved.

\end{proof}

Thanks to Lemma \ref{jan17 lemma:1} and Theorem XIII p. 74 of \cite{schwartz},
we can define the following distribution:
\begin{definition}\label{Aug12 def:1}
We set
$$
\int_0^\infty d_c\big(h(\frac s2,\cdot) \ast F(s,\cdot)\big)\, ds
:= \lim_{M\to\infty}
\int_0^M d_c\big(h(\frac s2,\cdot) \ast F(s,\cdot)\big)\, ds
$$
in $\mc D'(\he n,E_0^\bullet)$, where
$$
\Scal{\int_0^M d_c\big(h(\frac s2,\cdot) \ast F(s,\cdot)\big)}{\phi}\, ds
:= \int_0^M \Scal{d_c\big(h(\frac s2,\cdot) \ast F(s,\cdot)\big)}{\phi}\, ds
$$
for all $\phi\in \mc D(\he n, E_0^\bullet)$. 
\end{definition}

\subsection{The Calder\'on reproducing formula}

 If $\alpha\in L^1(\he n, E_0^h)$, $d_c\alpha=0$, let us set
\begin{equation}\label{Aug13 eq:1}
F(s,x):= d_c^*\big(h(\frac{s}{2}, \cdot)\ast \alpha\big)(x)\qquad s>0.
\end{equation}
By \eqref{Aug11 eq:3}, for any $s>0$ $F(s,\cdot)\in \mc O_M\subset \mc E\cap \mc S'$.
In particular $F(s,\cdot)$ is smooth  for any $s>0$. In addition, again by \eqref{Aug11 eq:3}, we can write
\begin{equation*}\begin{split}
F(s,\cdot) = h(\frac{s}{2}, \cdot)\ast d_c^*\alpha.
\end{split}\end{equation*}
If $\alpha = \sum_j \alpha_j\xi_j^h\in L^1(\he n,E_0^h)$, there exist homogeneous differential operators in the horizontal derivative $P_{j,\ell}${,say 
of order 1 or 2 according to the degree of the forms, such that
(with the formal notation of \eqref{Aug11 eq:4})
$$
d_c^*\alpha = \sum_{j,\ell} \big( P_{j,\ell}\alpha_j\big) (\xi_\ell^{h-1})^*.
$$}

\begin{theorem}If $\alpha \in L^1(\he n,E_0^h)$, we have:
\begin{equation}\label{reproducing}
\alpha = -\int_0^\infty d_c \big(h(\frac{s}{2}, \cdot)\ast F(s,\cdot) \big)\, ds
\end{equation}
\end{theorem}

\begin{proof} Since both $\alpha$ and 
$\int_0^\infty  d_c  \big(h(\frac{s}{2}, \cdot)\ast F(s,\cdot) \big)\, ds$ 
belong to $\mc D'(\he n, E_0^h)$ (see Definition \ref{Aug12 def:1}), 
it will be enough to show that, if $\phi \in \mc D(\he n, E_0^h)$,
then
\begin{equation}\label{Aug12 eq:2}\begin{split}
\Scal{\alpha}{\phi} &= - \Scal{\int_0^\infty  d_c  \big(h(\frac{s}{2}, \cdot)\ast F(s,\cdot) \big)\, ds}{\phi}
\\&
:=
- \lim_{M\to\infty} \int_0^M \Scal{d_c\big(h(\frac s2,\cdot) \ast F(s,\cdot)\big)}{\phi}\, ds.
\end{split}\end{equation}

Suppose first that $\alpha\in \mc D(\he n)$.
{If $h\neq n+1$, we have}:
\begin{equation}\label{Aug21 eq:7}\begin{split}
 \int_0^M  & \Scal{d_c\big(h(\frac s2,\cdot) \ast F(s,\cdot)\big)}{\phi}\, ds
 \\&
 =
  \int_0^M   \Scal{\big(h(\frac s2,\cdot) \ast F(s,\cdot)\big)}{d_c^*\phi}\, ds
   \\&
 =
  \int_0^M   \Scal{F(s,\cdot)}{\ccheck h(\frac s2,\cdot) \ast  d_c^*\phi}\, ds
  \qquad\mbox{(by \eqref{jan16 eq:6 bis})}
    \\&
: =
  \int_0^M   \Scal{h(\frac{s}{2}, \cdot)\ast \alpha}{\ccheck h(\frac s2,\cdot) \ast  d_cd_c^*\phi}\, ds
  \qquad\mbox{(by \eqref{Aug13 eq:1})}
      \\&
 =
  \int_0^M   \Scal{ \alpha}{\ccheck h(\frac{s}{2}, \cdot)\ast\ccheck h(\frac s2,\cdot) \ast  d_cd_c^*\phi}\, ds
  \qquad\mbox{(again by \eqref{jan16 eq:6 bis})}
        \\&
 =
  \int_0^M    \Scal{\alpha }{\ccheck h(s, \cdot)\ast (d_c d_c^*+d_c^*d_c) \phi}\, ds  \qquad\mbox{(since $d_c\alpha=0$}) 
   \\&
: =
  \int_0^M    \Scal{\alpha }{\ccheck h(s, \cdot)\ast \Delta_{\he{},h} \phi}\, ds  \qquad\mbox{(since $d_c\alpha=0$}) 
   \\&
: =
  \int_0^M    \int_{\he n}\scal{\alpha }{\ccheck h(s, \cdot)\ast \Delta_{\he{},h} \phi}\, dp\, ds,
\end{split}\end{equation}
since $\alpha\in L^1(\he n) $ and {$\ccheck h(s, \cdot)* \Delta_{\he{},h} \phi\in \mc S(\he n)$
 (if $h=n+1$ we must replace $d_c^*d_c$ with $(d_c^*d_c)^2$ to obtain the homogeneous Laplacian)}.

 We notice now that, arguing as in the proof of Lemma \ref{jan17 lemma:1},
\begin{equation}\label{Aug14 eq:1*}
{\scal {\alpha }{\ccheck h(s, \cdot)\ast \Delta_c \phi} \in L^1([0,\infty))\times\he n},
\end{equation}
since
\begin{equation}\label{Aug24 eq:1}
\int_0^\infty\int_{\he n} |\alpha| \cdot |\ccheck h(s, \cdot)\ast \Delta_c \phi|\, dp\,ds <\infty.
\end{equation}

%Let us write $\alpha=\sum_j\alpha_j\xi_j$. If $N\in\mathbb N$, denote now by $\alpha_N$
%the Rumin form with coefficients $\alpha_{j,N}$ defined by
%$$
%\alpha_{j,N} := \min\{|\alpha_j|\chi_{B(e,N)}, N\}.
%$$
%By definition, $\supp \alpha_N\subset B(e,N)$.
%Clearly, in addition, $ 0\le \alpha_{j,n}\le |\alpha|$, and $\alpha_N\in L^1(\he n, E_0^h)\cap L^2(\he n, E_0^h)$,
%so that, by dominated convergence theorem,
%$$
%  \int_0^M    \int_{\he n}\scal{\alpha }{\ccheck h(s, \cdot)\ast \Delta_{\he{},h} \beta}\, dp\, ds
%  = \lim_{N\to\infty}   \int_0^M    \int_{\he n}\scal{\alpha_N }{\ccheck h(s, \cdot)\ast \Delta_{\he{},h} \beta}\, dp\, ds.
%$$
%

Thus, by Fubini's theorem
 \begin{equation}\label{Aug20 eq:2*}\begin{split}
   \int_0^M   & \int_{\he n}\scal{\alpha }{\ccheck h(s, \cdot)\ast \Delta_{\he{},h} \phi}\, dp\, ds
   \\&
   =
       \int_{\he n}\scal{\alpha }{ \int_0^M \ccheck h(s, \cdot)\ast \Delta_{\he{},h} \phi\, ds}\, dp
       \\&
       = \Scal{ \int_0^M \ccheck h(s, \cdot)\ast \Delta_{\he{},h} \phi\, ds}{\alpha }.
 \end{split}\end{equation}
 
Let us write \eqref{Aug20 eq:2*} in terms of components. We get
 \begin{equation}\label{Aug21 eq:1}
\sum_{i,j,\ell} \int_{\he n}\alpha_i \int_0^M \ccheck h_{i,j}(s, \cdot)\ast \Delta_{\he{},h}^{j,\ell} \phi_\ell\, ds\, dp.
\end{equation}
We want to prove that
 \begin{equation}\label{Aug21 eq:2}\begin{split}
\int_{\he n}&\alpha_i \int_0^M \ccheck h_{i,j}(s, \cdot)\ast \Delta_{\he{},h}^{j,\ell} \phi_\ell\, ds\, dp
\\&
=\int_{\he n}\alpha_i \big(\int_0^M \ccheck h_{i,j}(s, \cdot)\,ds\big) \ast \Delta_{\he{},h}^{j,\ell} \phi_\ell\, dp,
\end{split}\end{equation}
all integrals in \eqref{Aug21 eq:2} being well defined.

For any $s>0$, since $h_{i,j}(s,\cdot)\in\mc S(\he n)$, we can write
\begin{equation}\label{Aug21 eq:3}\begin{split}
\big(\ccheck h_{i,j}(s, \cdot)  &  \ast \Delta_{\he{},h}^{j,\ell} \phi_\ell\big)(p)
\\&
= \int_{\he n} \ccheck h_{i,j}(s, q) \big(\Delta_{\he{},h}^{j,\ell} \phi_\ell)(q^{-1}p)  \, dq
\\&
= \int_{\he n} \ccheck h_{i,j}(s, q) \ccheck \big(\Delta_{\he{},h}^{j,\ell} \phi_\ell)(p^{-1}q)  \, dq
\\&
= \int_{\he n} \ccheck h_{i,j}(s, q) \ccheck \big(\Delta_{\he{},h}^{j,\ell} (\phi_\ell\circ \tau_{p^{-1}})\big)(q)  \, dq.
\end{split}\end{equation}
Now, by \eqref{Aug21 eq:5}, if $k>Q/2a$,
\begin{equation}\label{Aug21 eq:4}\begin{split}
\int_0^M\,ds  & \int |\ccheck h_{i,j}(s, q) |  \cdot |\ccheck \big(\Delta_{\he{},h}^{j,\ell} (\phi_\ell\circ \tau_{p^{-1}})\big)(q) | \, dq
\\&
\le C\, M\, \| \big(\Delta_{\he{},h}^{j,\ell} (\phi_\ell\circ \tau_{p^{-1}})\big)\|_{W^{2,k}(\he n)}
\\&
\le C\, M\, \sum_{|I| \le k+a} \| \phi_\ell\circ \tau_{p^{-1}}\|_{L^2(\he n)}
\\&
= C\, M\, \sum_{|I| \le k+a} \| \phi_\ell  \|_{L^2(\he n)}<\infty.
\end{split}\end{equation}
Thus, combining \eqref{Aug21 eq:3} and \eqref{Aug21 eq:4}, the map
$$
(s,p)\to \big(\ccheck h_{i,j}(s, \cdot)    \ast \Delta_{\he{},h}^{j,\ell} \phi_\ell\big)(p)
$$
belongs to $L^1([0,M]\times \supp \alpha_i)$ and, by Fubini theorem,
\begin{equation}\label{Aug21 eq:6}\begin{split}
\int_{\he n} \alpha_i(p) & \int_0^M \big(\ccheck h_{i,j}(s, \cdot)    \ast \Delta_{\he{},h}^{j,\ell} \phi_\ell\big)(p)\, ds \, dp
\\&
=
\int_{\he n} \alpha_i(p) \int_0^M
\int_{\he n} \ccheck h_{i,j}(s, q) \ccheck \big(\Delta_{\he{},h}^{j,\ell} (\phi_\ell\circ \tau_{p^{-1}})\big)(q)  \, dq  \, ds \, dp
\\&
=
\int_{\he n} \alpha_i(p) 
\int_{\he n} \Big\{ \int_0^M \ccheck h_{i,j}(s, q)\,ds\Big\} \ccheck \big(\Delta_{\he{},h}^{j,\ell} (\phi_\ell\circ \tau_{p^{-1}})\big)(q)  \, dq  \, dp
\\&
=
\int_{\he n} \alpha_i(p) 
\Big(\Big\{ \int_0^M \ccheck h_{i,j}(s, \cdot)\,ds\Big\} \ast \Delta_{\he{},h}^{j,\ell} \phi_\ell\Big)(p)   \, dp
\\&
= \Scal{\Big\{ \int_0^M \ccheck h_{i,j}(s, \cdot)\,ds\Big\} \ast \Delta_{\he{},h}^{j,\ell} \phi_\ell}{\alpha_i}.
\end{split}\end{equation}
Thus \eqref{Aug20 eq:2*} becomes
 \begin{equation}\label{Aug20 eq:2**}\begin{split}
   \int_0^M   & \int_{\he n}\scal{\alpha }{\ccheck h(s, \cdot)\ast \Delta_{\he{},h} \phi}\, dp\, ds
   \\&
   =
  \Scal{\Big\{ \int_0^M \ccheck h(s, \cdot)\,ds\Big\} \ast \Delta_{\he{},h}\phi}{\alpha}.
 \end{split}\end{equation}
 On the other hand, by Proposition \ref{Aug11 prop:1}, we know that
 $$
\int_0^M \ccheck h(s, \cdot)\,ds \to \int_0^\infty \ccheck h(s, \cdot)\,ds\qquad\mbox{in $\mc D'(\he n, E_0^h)$.}
 $$
 But the map $T\to T\ast \psi$ is continuous from $\mc D'(\he n)$ to $\mc D'(\he n)$ for fixed $\psi\in\mc D(\he n)$
 (this is a special instance of Th\'eor\`eme V, p.157 of \cite{schwartz}), so that 
 \begin{equation}\label{Aug21 eq:8}
  \Scal{\Big\{ \int_0^M \ccheck h(s, \cdot)\,ds\Big\} \ast \Delta_{\he{},h}\phi}{\alpha}
  \longrightarrow
   \Scal{\Big\{ \int_0^\infty \ccheck h(s, \cdot)\,ds\Big\} \ast \Delta_{\he{},h}\phi}{\alpha}
\end{equation}
as $M\to\infty$. Thus, combining \eqref{Aug21 eq:7}, \eqref{Aug20 eq:2**} and \eqref{Aug21 eq:8},
the reproducing formula \eqref{reproducing} is proved when $\alpha\in \mc D(\he n,E_0^h)$

\medskip

Eventually, let us consider the case $\alpha\in L^1(\he n,E_0^h)$. Suppose for a while we are
able to prove the reproducing formula \eqref{reproducing} when $\alpha$ is replaced by
a form $\tilde\alpha\in L^2(\he n, E_0^h)\cap \mc E'(\he n, E_0^h)$. We argue as follows:
if $\alpha = \sum_j\alpha_j \xi j$ and $N\in\mathbb N$, we set
$$
\alpha_{N,j}:= \min\{|\alpha_j|\chi_{B(e,N)}, N\}\, \frac{\alpha_j}{|\alpha_j|}, \qquad\mbox{where $ \alpha\neq 0$}
$$
and
$$
\alpha_N = \sum_j \alpha_{N,j}\xi j.
$$
Since $\alpha_{N,j}$ is compactly supported and bounded, then 
{$\alpha_N\in L^1(\he n,E_0^h)$}. In
addition, a.e. $\alpha_{N,j}\to\alpha_j$ as $N\to\infty$, and $|\alpha_{N,j}|\le|\alpha_j|$, $j=1,\dots, N_h$.

%Arguing as in \cite{Aug21 eq:7}, if $\phi\in \mc D(\he n,E_0^h)$, then
%$$
%\Scal{d_c\big(h(\frac s2,\cdot) \ast F_N(s,\cdot)\big)}{\phi} =  \int_{\he n}\scal{\alpha_N }{\ccheck h(s, \cdot)\ast \Delta_{\he{},h} \phi}\, dp
%$$
%for a.e. $s>0$. In addition, as in \eqref{Aug24 eq:1}
%\begin{equation}\label{Aug24 eq:2}
%\int_0^\infty\int_{\he n} |\alpha-\alpha_N| \cdot |\ccheck h(s, \cdot)\ast \Delta_c \phi|\, dp\,ds <\infty.
%\end{equation}
In addition, set
\begin{equation}\label{Aug13 eq:1*}
F_N(s,x):= d_c^*\big(h(\frac{s}{2}, \cdot)\ast \alpha_N\big)(x)\qquad s>0.
\end{equation}
By our temporary assumption, if $\phi$ is a test form, arguing as in \eqref{Aug21 eq:7}, we get
\begin{equation}\label{reproducing*}\begin{split}
\scal{\alpha_N }{\phi}&= - \scal{\int_0^\infty d_c \big(h(\frac{s}{2}, \cdot)\ast F_N(s,\cdot) \big)\, ds}{\phi}
\\&
=  \int_0^\infty    \int_{\he n}\scal{\alpha_N }{\ccheck h(s, \cdot)\ast \Delta_{\he{},h} \phi}\, dp\, ds.
\end{split}\end{equation}
Since $|\alpha_N|\le|\alpha|$, we can take the limit in \eqref{reproducing*} as $N\to\infty$ and we
obtain \eqref{reproducing}.

Thus, we are left with the case 
$$
\alpha=\sum_j\alpha_j\xi_j\qquad\mbox{with $\alpha_j\in L^2(\he n)\cap \mc E'(\he n)$.}
$$

If $(\omega_\eps)_{\eps>0}$ are the (usual)
 Friedrichs' mollifiers, we set
 $$
 \alpha_{j,\eps} :={\alpha_j}\ast \omega_\eps\in \mc D({\he n},E_0^h),
 $$
 and
 $$
 \alpha_\eps \sum_j \alpha_{j,\eps} \xi_j \in \mc D(\he n,E_0^h).
 $$
Denote now by $\gamma\in E_0^h$ the Rumin's form 
 $$
 \sum_j (M \alpha_{j,\eps}) \xi_j,
 $$
 where $M$ is the Hardy-Littlewood maximal function. It is well known that
 \begin{equation}\label{Aug20 eq:1*}
 |\alpha_{j,\eps} |  \le  \gamma_{j}\qquad\mbox{a.e. in $\he n$  for $j=1,\dots,N_h$.}
\end{equation}
Moreover, since $\alpha\in L^2(\he n, E_0^h)$ then 
$$
\gamma \in L^2(\he n, E_0^h).
$$
Let us prove now that 
\begin{equation}\label{Aug19 eq:1*}
|\gamma|\cdot|\ccheck h(s, \cdot)\ast \Delta_{\he{},h} \phi| \in L^1([0,\infty)\times \he n).
\end{equation}
Indeed we have:
\begin{equation*}\begin{split}
\int_0^\infty\int_{\he n} & |\gamma| \cdot|{\ccheck h(s, \cdot)* \Delta_{\he{},h} \phi|)}\, dp\,ds
\\&
\le
\| \gamma\|_{L^2(\he n)} \cdot\int_0^\infty \| \ccheck h(s, \cdot)\ast \Delta_{\he{},h} \phi\|_{L^2(\he n)}\,ds
\\&
=
\| \gamma\|_{L^2(\he n)} \Big( \int_0^1 \cdots \,ds+
\int_1^\infty \cdots \,ds\Big).
\end{split}\end{equation*}
Now
\begin{equation*}\begin{split}
 \int_0^1 &\| \ccheck h(s, \cdot)\ast \Delta_{\he{},h} \phi\|_{L^2(\he n)} \,ds
  \\&= 
  \int_0^1 \| \exp(-s\Delta_{\he{},h})\ccheck\Delta_{\he{},h} \phi\|_{L^2(\he n)} \,ds
 \le C_\phi\qquad\mbox{(by \eqref{strongly}),}
\end{split}\end{equation*}
whereas, keeping in mind that $h(1,\cdot) \in \mc S$,
\begin{equation*}\begin{split}
\int_1^\infty &\| \ccheck h(s, \cdot)\ast \Delta_{\he{},h} \phi\|_{L^2(\he n)} \,ds
\\&
=
\sup_{\he n} | h(1, \cdot) | \cdot \int_1^\infty s^{-Q/a} \|  \Delta_{\he{},h} \phi\|_{L^2(\he n)} \,ds <\infty
\end{split}\end{equation*}
Then we can write \eqref{reproducing} for ${\alpha_\eps}\in\mc D(\he n , E_0^h)$ and 
(by dominate convergence theorem) take the limit
as $\eps\to 0$. This completes the proof of the theorem.
\end{proof}

\section{Appendix A: Rumin's complex on Heisenberg groups}\label{rumin heisenberg}
In this appendix, we present some basic notations and introduce both the structure of Heisenberg groups together with the formulation of the Rumin complex.  We denote by  $\he n$  the $(2n+1)$-dimensional Heisenberg
group, identified with $\rn {2n+1}$ through exponential
coordinates. A point $p\in \he n$ is denoted by
$p=(x,y,t)$, with both $x,y\in\rn{n}$
and $t\in\R$.
   If $p$ and
$p'\in \he n$,   the group operation is defined by
\begin{equation*}
p\cdot p'=(x+x', y+y', t+t' + \frac12 \sum_{j=1}^n(x_j y_{j}'- y_{j} x_{j}')).
\end{equation*}
{Notice that $\he n$ can be equivalently identified with $\mathbb C^{n}\times \mathbb R$
endowed with the group operation
$$
(z,t)\cdot (\zeta,\tau): = (z+\zeta, t+\tau - \frac12\,Im\,(z\bar{\zeta})).
$$
} 
The unit element of $\he n$ is the origin, that will be denote by $e$.
For
any $q\in\he n$, the {\it (left) translation} $\tau_q:\he n\to\he n$ is defined
as $$ p\mapsto\tau_q p:=q\cdot p. $$

For a general review on Heisenberg groups and their properties, we
refer to \cite{Stein}, \cite{GromovCC} and to \cite{VarSalCou}. See also \cite{FSSC_advances} for notations. 

%{  First we notice that Heisenberg groups are smooth manifolds (and therefore
%are Lie groups). In particular, the pull-back of differential forms is well 
%defined as follows (see, e.g. \cite{GHL}, Proposition 1.106);
%\begin{definition} If $\;\mc U,\mc V$ are open subsets of $\he n$, and $f: \mc U\to
%\mc V$ is a
%diffeomorphism, then for any  differential form  $\alpha$ of degree $h$,
% we denote by $f^\sharp \alpha$ the pull-back form in $\mc
%\mc U$ defined by
%$$
%(f^\sharp \alpha)(p) (v_1,\dots,v_h):=  \alpha(f(p)) (df(p)v_1,\dots,df(p)v_h)
%$$
%for any $h$-tuple $(v_1,\dots,v_h)$ of tangent vectors at $p$.
%
%\end{definition}
%}

The Heisenberg group $\he n$ can be endowed with {a homogeneous} %the homogeneous
norm (Cygan-Kor\'anyi norm): if $p=(x,y,t)\in \he n$, then we set
\begin{equation}\label{gauge}
\varrho (p)=\big((x^2+y^2)^2+ 16 t^2\big)^{1/4},
\end{equation}
and we define the gauge distance (a true distance, see
 \cite{Stein}, p.\,638, with a different normalization in the group law), that is left invariant i.e. $d(\tau_q p,\tau_q p')=d(p,p' )$ for all $p,p'\in\he n$)
as
\begin{equation}\label{def_distance}
d(p,q):=\varrho ({p^{-1}\cdot q}).
\end{equation}
{Notice that $d$ is equivalent to the Carnot-Carath\'eodory distance on $\he n$ (see, e.g., \cite{BLU}, Corollary 5.1.5).}
Finally, the balls for the metric $d$ are the so-called Cygan-Kor\'anyi balls
\begin{equation}\label{koranyi}
B(p,r):=\{q \in  \he n; \; d(p,q)< r\}.
\end{equation}

Notice that Cygan-Kor\'anyi balls are convex smooth sets. A straightforward computation shows that, if $ \rho(p) <1$, then
\begin{equation}\label{c0}
 |p| \le \rho(p) \le |p|^{1/2}.
\end{equation}

It is well known that the topological dimension of $\he n$ is $2n+1$,
since as a smooth manifold it coincides with $\R^{2n+1}$, whereas
the Hausdorff dimension of $(\he n,d)$ is $Q:=2n+2$ 
(the so called \emph{homogeneous dimension} of $\he n$).

    We denote by  $\mfrak h$
 the Lie algebra of the left
invariant vector fields of $\he n$. The standard basis of $\mfrak
h$ is given, for $i=1,\dots,n$,  by
\begin{equation*}
X_i := \partial_{x_i}- \frac12 y_i \partial_{t},\quad Y_i :=
\partial_{y_i}+\frac12 x_i \partial_{t},\quad T :=
\partial_{t}.
\end{equation*}
The only non-trivial commutation  relations are $
[X_{j},Y_{j}] = T $, for $j=1,\dots,n.$ 
The {\it horizontal subspace}  $\mfrak h_1$ is the subspace of
$\mfrak h$ spanned by $X_1,\dots,X_n$ and $Y_1,\dots,Y_n$:
${ \mfrak h_1:=\mathrm{span}\,\left\{X_1,\dots,X_n,Y_1,\dots,Y_n\right\}\,.}$

\noindent Coherently, from now on, we refer to $X_1,\dots,X_n,Y_1,\dots,Y_n$
(identified with first order differential operators) as
the {\it horizontal derivatives}. Denoting  by $\mfrak h_2$ the linear span of $T$, the $2$-step
stratification of $\mfrak h$ is expressed by
\begin{equation*}
\mfrak h=\mfrak h_1\oplus \mfrak h_2.
\end{equation*}

\bigskip

{ 
The stratification of the Lie algebra $\mfrak h$ induces a family of non-isotropic dilations 
$\delta_\lambda: \he n\to\he n$, $\lambda>0$ as follows: if
$p=(x,y,t)\in \he n$, then
\begin{equation}\label{dilations}
\delta_\lambda (x,y,t) = (\lambda x, \lambda y, \lambda^2 t).
\end{equation}
}

{
\begin{remark} Heisenberg groups are special instance of the so-called
{\it Carnot groups}.  A \emph{graded group} of step $\mfrak g$   
is a connected, simply connected
Lie group $\G$ whose finite dimensional  Lie algebra $\mfrak g$ is the direct sum of $\kappa$ subspaces $\mfrak g_i$, 
$\mfrak g=\mfrak g_1\oplus\cdots\oplus  \mfrak g_\kappa$,
 such that
 \begin{equation*}
\left[\mfrak g_i, \mfrak g_j\right]\subset \mfrak g_{i+j},\quad \text{for $1\leq i,j \leq \kappa$},
\end{equation*}
where $\mfrak g_i=0$ for $i>\kappa$.
We denote as $n$ the dimension of $\mathfrak g$ and as $n_j$ the dimension of $\mfrak g_j$, for $1\leq j\leq \kappa $.

A \emph {Carnot group} $\G$ of
step $\kappa$ is a graded group of step $\mfrak g$, where $\mfrak g_1$ generates all of $\mfrak g$. That is $
\left[\mfrak g_1, \mfrak g_i\right]= \mfrak g_{i+1},
$ for $i=1,\dots,\kappa$.  We refrain from dealing with such generality. 
\end{remark}
}

Going back to Heisenberg groups, the vector space $ \mfrak h$  can be
endowed with an inner product, {denoted} by
$\scalp{\cdot}{\cdot}{} $,  making
    $X_1,\dots,X_n$,  $Y_1,\dots,Y_n$ and $ T$ orthonormal.
    
Throughout this paper, we write also
\begin{equation}\label{campi W}
W_i:=X_i, \quad W_{i+n}:= Y_i\quad { \mathrm{and} } \quad W_{2n+1}:= T, \quad \text
{for }i =1, \dots, n.
\end{equation}
Following \cite{folland_stein}, we also adopt the following
multi-index notation for higher-order derivatives. If $I =
(i_1,\dots,i_{n})$ is a multi--index, we set
\begin{equation}\label{centrata}
W^I=W_1^{i_1}\cdots
W_{n}^{i_{n}}.
\end{equation}
\begin{remark}\label{PBW} By the Poincar\'e--Birkhoff--Witt
theorem (see, e.g. \cite{bourbaki}, I.2.7), the differential
operators $W^I$ form a basis for the algebra of left invariant
differential operators {on} $\G$. Furthermore, we
{denote by}
$|I|:=i_1+\cdots +i_{n}$ the order of the differential
operator $W^I$, and {by}  $d(I):=d_1i_1+\cdots +d_n i_{n}$ its
degree of homogeneity with respect to group dilations.
 From the Poincar\'e--Birkhoff--Witt theorem, it follows, in particular, that any
homogeneous linear differential operator in the horizontal
derivatives can be expressed as a linear combination of the
operators $W^I$ of the special form above. Thus,  often we can
restrict ourselves to consider only operators of the special form
$W^I$.
\end{remark}
The dual space of $\mfrak h$ is denoted by $\covH 1$.  The  basis of
$\covH 1$,  dual to  the basis $\{X_1,\dots , Y_n,T\}$,  is the family of
covectors $\{dx_1,\dots, dx_{n},dy_1,\dots, dy_n,\theta\}$ where 
\begin{equation}\label{theta}
 \theta
:= dt - \frac12\, \sum_{j=1}^n (x_jdy_j-y_jdx_j)
\end{equation}
 is called the {\it contact
form} in $\he n$. 
We {also} denote by $\scalp{\cdot}{\cdot}{} $ the
inner product in $\covH 1$  that makes $(dx_1,\dots, dy_{n},\theta  )$ 
an orthonormal basis.

Coherently with the previous notation \eqref{campi W},
we set
\begin{equation*}
\omega_i:=dx_i, \quad \omega_{i+n}:= dy_i \quad { \mathrm{and} }\quad \omega_{2n+1}:= \theta, \quad \text
{for }i =1, \dots, n.
\end{equation*}

{ 
We put
$       \vetH 0 := \covH 0 =\R $
and, for $1\leq h \leq 2n+1$,
\begin{equation*}
\begin{split}
         \covH h& :=\mathrm {span}\{ \omega_{i_1}\wedge\dots \wedge \omega_{i_h}:
1\leq i_1< \dots< i_h\leq 2n+1\}.
\end{split}
\end{equation*}
In the sequel we shall denote by $\Theta^h$ the basis of $ \covH h$ defined by
$$
\Theta^h:= \{ \omega_{i_1}\wedge\dots \wedge \omega_{i_h}:
1\leq i_1< \dots< i_h\leq 2n+1\}.
$$
To avoid cumbersome notations, if $I:=({i_1},\dots,{i_h})$, we write
$$
\omega_I := \omega_{i_1}\wedge\dots \wedge \omega_{i_h}.
$$
The  inner product $\scal{\cdot}{\cdot}$ on $ \covH 1$ yields naturally an inner product 
$\scal{\cdot}{\cdot}$ on $ \covH h$
making $\Theta^h$ an orthonormal basis.

The volume $(2n+1)$-form $ \omega_1\wedge\cdots\wedge \omega_{ 2n+1}$
 will be also
written as $dV$.

Throughout this paper, the elements of $\cov h$ are identified with \emph{left invariant} differential forms
of degree $h$ on $\he n$. 

\begin{definition}\label{left} A $h$-form $\alpha$ on $\he n$ is said left invariant if 
$$\tau_q^\#\alpha
=\alpha\qquad\mbox{for any $q\in\he n$.}
$$
\end{definition}

The pull-back of differential forms is well 
defined as follows (see, e.g. \cite{GHL}, Proposition 1.106);
\begin{definition} If $\;\mathcal U,\mathcal V$ are open subsets of $\mathbb H^n$, and $f: \mathcal U\to
\mathcal V$ is a
diffeomorphism, then for any  differential form  $\alpha$ of degree $h$,
 we denote by $f^\sharp \alpha$ the pull-back form in $\mathcal
U$ defined by
$$
\Scal{f^\sharp \alpha(p)}{ v_1,\dots,v_h}:=  {\alpha(f(p))}{ df(p)v_1,\dots,df(p)v_h}
$$
for any $h$-tuple $(v_1,\dots,v_h)$ of tangent vectors at $p$.

\end{definition}

The same construction can be performed starting from the vector
subspace $\mfrak h_1\subset \mfrak h$,
obtaining the {\it horizontal $h$-covectors} 

\begin{equation*}
\begin{split}
         \covh h& :=\mathrm {span}\{ \omega_{i_1}\wedge\dots \wedge \omega_{i_h}:
1\leq i_1< \dots< i_h\leq 2n\}.
\end{split}
\end{equation*}
It is easy to see that 
$$
\Theta^h_0 := \Theta^h \cap  \covh h
$$ 
provides an orthonormal
basis of $ \covh h$.

Keeping in mind that the Lie algebra $\mathfrak h$ can be identified with the
tangent space to $\he n$ at $x=e$ (see, e.g. \cite{GHL}, Proposition 1.72), 
starting from $\cov h$ we can define by left translation  a fiber bundle
over $\he n$  that we can still denote by $\cov h$. We can think of $h$-forms as sections of 
$\cov h$. We denote by $\Omega^h$ the
vector space of all smooth $h$-forms.

\bigskip

The stratification
of the Lie algebra $\mfrak h$ yields a lack of homogeneity of de Rham's exterior differential
with respect to group dilations $\delta_\lambda$.  Thus, to keep into account the different degrees
of homogeneity of the covectors when they vanish on different layers of the
stratification, we introduce the notion of {\sl weight} of a covector as follows.
}

\begin{definition}\label{weight} If $\eta\neq 0$, $\eta\in \covh 1$,  
 we say that $\eta$ has \emph{weight $1$}, and we write
$w(\eta)=1$. If $\eta = \theta$, we say $w(\eta)= 2$.
More generally, if
$\eta\in \covH h$, {  $\eta\neq 0$, }we say that $\eta$ has \emph {pure weight} $p$ if $\eta$ is
a linear combination of covectors $\omega_{i_1}\wedge\cdots\wedge\omega_{i_h}$
with $w(\omega_{i_1})+\cdots + w(\omega_{ i_h})=p$.
\end{definition}

{  Notice that, if $\eta,\zeta \in \covH h$ and $w(\eta)\neq w(\zeta)$, then
$\scal{\eta}{\zeta}=0$ (see \cite{BFTT}, Remark 2.4). Also, we point out that $w(d\theta)=w(\theta)$, since,  if $\alpha$
is a left invariant 
$h$-form of weight $p$ and $d\alpha\neq 0$,
then $w(d\alpha)=w(\alpha)$ ( See \cite{rumin_grenoble}, Section 2.1). 

We stress that generic covectors may fail to have a pure weight: it is enough to
consider $\he 1$ and the covector $dx_1+\theta\in \covH{1}$. However, the
following result holds
(see \cite{BFTT}, formula (16)):
\begin{equation}\label{dec weights}
\covH h = \covw {h}{h}\oplus \covw {h}{h+1} =  \covh h\oplus \Big(\covh {h-1}\Big)\wedge \theta,
\end{equation}
where $\covw {h}{p}$ denotes the linear span of the $h$-covectors of weight $p$.
By our previous remark, the decomposition \eqref{dec weights} is orthogonal.
In addition, since the elements of the basis $\Theta^h$ have pure weights, a basis of
$ \covw {h}{p}$ is given by $\Theta^{h,p}:=\Theta^h\cap \covw {h}{p}$
(such a basis is usually called an adapted basis). 

{We notice that, according to \eqref{dec weights}, the weight of a $h$-form
is either $h$ or $h+1$ and there are no {$h$-forms} of weight $h+2$, since  there
is only one 1-form of weight 2. Something analogous can be possible for instance in
$\he n\times \mathbb R$, but it fails to be possible already in the case of general step 2 groups
with higher dimensional center.
}

As above, starting from  $\covw {h}{p}$, we can define by left translation  a fiber bundle
over $\he n$  that we can still denote by $\covw {h}{p}$. 
Thus, if we denote by  $\Omega^{h,p} $ the vector space of all
smooth $h$--forms in $\he n$ of  weight $p$, i.e. the space of all
smooth sections of $\covw {h}{p}$, we have
\begin{equation}\label{deco forms}
\Omega^h = \Omega^{h,h}\oplus\Omega^{h,h+1} .
\end{equation}

\bigskip

{ Starting from the notion of weight of a differential form, it is possible
to define a new complex of differential forms $(E_0^\bullet,d_c)$
that is homotopic to the de Rham's complex and respects the homogeneities
of the group. 
}

{We sketch here the construction of the Rumin complex. For a more detailed presentation we
refer to Rumin's papers ( see e.g. \cite{rumin_grenoble}). Here we follow the presentation of \cite{BFTT}.
}
The exterior differential $d$ does not preserve weights. It splits into
\begin{eqnarray*}
d=d_0+d_1+d_2
\end{eqnarray*}
where $d_0$ preserves weight, $d_1$ increases weight by 1 unit and $d_2$ increases weight by 2 units. 
{ 

More explicitly,
let $\alpha\in \Omega^{h}$ be a (say) smooth form
of pure weight $h$. We can write
$$
\alpha= \sum_{\omega_I\in\Theta^{h}_0}\alpha_{I}\, \omega_I 
,\quad
\mbox{with } \alpha_I \in \mc C^\infty (\he n).
$$
Then
$$
d\alpha= \sum_{\omega_I\in\Theta^{h}_0}\sum_{j=1}^{2n} (W_j\alpha_{I})\, \omega_j\wedge\omega_I + 
\sum_{\omega_I\in\Theta^{h}_0} (T \alpha_{I})\, \theta\wedge \omega_I = d_1\alpha +  d_2\alpha,
$$
and $d_0\alpha =0$. On the other hand, if $\alpha\in \Omega^{h,h+1}$ has pure weight $h+1$, then 
$$
\alpha =  \sum_{\omega_J\in\Theta^{h-1}_0}\alpha_{J}\, \theta\wedge\omega_J,
$$
and
$$
d\alpha= \sum_{\omega_J\in\Theta^{h}_0}\alpha_J\,d\theta\wedge\omega_J + \sum_{\omega_J\in\Theta^{h}_0}\sum_{j=1}^{2n} (W_j\alpha_{J})\, \omega_j\wedge\theta\wedge\omega_I 
=d_0\alpha+d_1\alpha,
$$
and $d_2\alpha=0$.

It is crucial to notice that  $d_0$ is an algebraic operator, in the sense that
for any real-valued $f\in\mc C^\infty (\he n)$ we have
$$
d_0(f\alpha)= f d_0\alpha,
$$
so that its action can be identified at any point with the action of a linear
operator from   $\cov h$ to $\cov {h+1}$ (that we denote again by $d_0$). 

Following M. Rumin (\cite{rumin_grenoble}, \cite{rumin_cras}) we give the following definition:
\begin{definition}\label{E0}
If $0\le h\le 2n+1$, keeping in mind that $\cov h$ is endowed with a canonical
inner product, we set
$$
E_0^h:= \ker d_0\cap (\mathrm{Im}\; d_0)^{\perp}.
$$
Straightforwardly, $E_0^h$ inherits from $\cov h$ the
inner product. 
\end{definition}

As above,  $E_0^\bullet$ defines by left translation a fibre bundle over $\he n$,
that we still denote by $E_0^\bullet$. To avoid cumbersome notations,
we denote also by  $E_0^\bullet$ the space of sections of this fibre bundle.

Let $L: \cov h \to \cov{h+2}$ the Lefschetz operator defined by
\begin{equation}\label{lefs}
L\, \xi = d\theta\wedge\xi.
\end{equation}
Then the spaces $E_0^\bullet$  can be defined explicitly as follows:

\begin{theorem}[see \cite{rumin_jdg}, \cite{rumin_gafa}] \label{rumin in H} We have:
\begin{itemize}
\item[i)] $E_0^1= \covh{1}$;
\item[ii)]  if $2\le h\le n$, then $E_0^h= \covh{h}\cap \big(\covh{h-2}\wedge d\theta\big)^\perp$
 (i.e. $E_0^h$ is the space of the so-called \emph{primitive covectors} of $\covh h$);
\item[iii)]  if $n< h\le 2n+1$, then $E_0^h = \{\alpha = \beta\wedge\theta, \; \beta\in \covh{h-1},
\; \gamma\wedge d\theta =0\} = \theta\wedge\ker L$;
\item[iv)]  if $1<h\le n$, then $N_h:=\dim E_0^h = \binom{2n}{h} - \binom{2n}{h-2}$;
\item[v)]  if $\ast$ denotes the Hodge duality associated with the inner product in $\cov{\bullet}$
and the volume form $dV$, then $\ast E_0^h = E_0^{2n+1-h}$.
\end{itemize} 
Notice that all forms in $E_0^h$ have weight $h$ if $1\le h\le n$ and
weight $h+1$ if $n< h\le 2n+1$.

\end{theorem}

A further geometric interpretation (in terms of decomposition of $\mathfrak h$ and of graphs
within $\he n$) can be found in \cite{FS2}.

Notice that there exists a left invariant  basis 
\begin{equation}\label{basis E0}
\Xi_0^h=\{\xi_1^h,\dots, \xi_{N_h}^h\}
\end{equation}
of $E_0^h$ that is adapted to the filtration
\eqref{dec weights}. Such a basis is explicitly constructed by {  induction} in \cite{BBF}
and in \cite{vittone2020}.
 To avoid cumbersome notations, if there is no risk of misunderstandings and the
 degrees $h$ of the forms is evident or uninfluential, we write $\xi_j$ for
 $\xi_j^h$.  

The core of Rumin's theory consists in the construction of a suitable ``exterior differential''
$d_c: E_0^h \to E_0^{h+1}$ making $\mc E_0:= (E_0^\bullet, d_c)$ a complex homotopic
to the de Rham complex.

Let us sketch Rumin's construction: first the next result (see \cite{BFTT}, Lemma 2.11 for a proof) allows us to define a (pseudo) inverse of $d_0$ : 
\begin{lemma}\label{d_0} If $1\le h\le n$, then $\ker d_0 = \covh{h}$.
Moreover, if $\beta\in \covH {h+1}$, then there exists a unique $\gamma\in
\covH{h}\cap (\ker d_0)^\perp$ such that
$$
d_0\gamma-\beta\in \mc R(d_0)^\perp.
$$

\end{lemma}
With the notations of the previous lemma, we set $$\gamma :=d_0^{-1}\beta.$$
We notice that $d_0^{-1}$ preserves the weights.

The following  theorem summarizes the construction of 
the intrinsic differential $d_c$ (for details, see \cite{rumin_grenoble}
and \cite{BFTT}, Section 2) .
\begin{theorem}\label{main rumin new}
The de Rham complex $(\Omega^\bullet,d)$ 
splits into the direct sum of two sub-complexes $(E^\bullet,d)$ and
$(F^\bullet,d)$, with
$$
E:=\ker d_0^{-1}\cap\ker (d_0^{-1}d)\quad\mbox{and}\quad
F:= \mc R(d_0^{-1})+\mc R (dd_0^{-1}).
$$
Let $\Pi_E$ be the projection on $E$ along $F$ (that
is not an orthogonal projection). We have
\begin{itemize}
\item[i)]   If $\gamma\in E_0^{h}$,  then
\begin{itemize}
\item[$\bullet$] $
\Pi_E\gamma=\gamma -d_0^{-1}
d_1 \gamma$ if $1\le h\le n$;
\item[$\bullet$] $
\Pi_E\gamma=\gamma $ if $h>n$.
\end{itemize}
\item[ii)] $\Pi_{E}$ is a chain map, i.e.
$$
d\Pi_{E} = \Pi_{E}d.
$$
\item[iii)] Let $\;\Pi_{E_0}$ be the orthogonal projection from $\covH{*}$
on $E_0^\bullet$, then
\begin{equation}\label{PiE0 project}
\Pi_{E_0} = I - d_0^{-1}d_0-d_0d_0^{-1}, \quad
\Pi_{E_0^\perp} =  d_0^{-1}d_0 + d_0d_0^{-1}.
\end{equation}

\item[iv)] $\Pi_{E_0}\Pi_{E}\Pi_{E_0}=\Pi_{E_0}$ and
$\Pi_{E}\Pi_{E_0}\Pi_{E}=\Pi_{E}$.
\end{itemize}

\noindent Set now
 $$d_c=\Pi_{E_0}\, d\,\Pi_{E}: E_0^h\to E_0^{h+1}, \quad h=0,\dots,2n.$$
 We have:
\begin{itemize}
\item[v)] $d_c^2=0$;
\item[vi)] the complex $E_0:=(E_0^\bullet,d_c)$ is homotopic to the de Rham complex;
\item[vii)] $d_c: E_0^h\to E_0^{h+1}$ is a homogeneous differential operator in the 
horizontal derivatives
of order 1 if $h\neq n$, whereas $d_c: E_0^n\to E_0^{n+1}$ is an homogeneous differential operator in the 
horizontal derivatives
of order 2.
\end{itemize}
\end{theorem}
}

 \begin{remark}\label{rumin carnot} The construction of Rumin's complex can be carried out
 in general Carnot groups; we refer for instance to \cite{rumin_palermo}, \cite{rumin_grenoble},
 \cite{BFTT}. The starting point is a notion of {\it weight} of a covector
 in term of homogeneity  with respect to group dilations. For an alternative
 presentation, we refer to \cite{FS2}.
 
 \end{remark}

Since the exterior differential $d_c$ on $E_0^h$ can be written in coordinates as a left-invariant homogeneous differential operator in the horizontal variables, of order 1 if $h\neq n$ and of order
2 if $h=n$, the proof of the following Leibniz' formula is easy.

{ 
\begin{lemma}\label{leibniz} If $\zeta$ is a smooth real function, then
\begin{itemize}
\item if $h\neq n$, then on $E_0^h$ we have:
$$
[d_c,\zeta] = P_0^h,
$$
where $P_0^h: E_0^h \to E_0^{h+1}$ is a linear homogeneous differential operator of degree zero, with coefficients depending
only on the horizontal derivatives of $\zeta$;
\item if $h= n$, then on $E_0^n$ we have
$$
[d_c,\zeta] = P_1^n + P_0^n ,
$$
where $P_1^n: E_0^n \to E_0^{n+1}$ is a linear homogeneous differential operator of degree 1, with coefficients depending
only on the horizontal derivatives of $\zeta$, and where $P_0^h: E_0^n \to E_0^{n+1}$ is a linear homogeneous differential operator in
the horizontal derivatives of degree 0
 with coefficients depending
only on second order horizontal derivatives of $\zeta$.
\end{itemize}

\end{lemma}
}

\section{ Appendix B: Kernels in Carnot groups and Folland-Stein Spaces}\label{kernels carnot}

%The following sections deal
%with Sobolev spaces (the so-called Folland-Stein-Sobolev spaces: see \cite{folland}, \cite{folland_stein})),
%and with the calculus for homogeneous kernels (see \cite{christ_et_al}) 
%in the more general setting of Carnot groups. Heisenberg groups will
%provide special instance. We refer to \cite{folland,folland_stein} for the standard definitions of Sobolev spaces and their H\"older versions.  
%

\subsection{Convolution in Carnot groups}

{If $f:\he n\to\mathbb R$, we set $\ccheck f(p) = f(p^{-1})$, and,
if $T\in \mc D'(\he n)$, then $\Scal{\ccheck T}{\phi} := \Scal{T}{\ccheck \phi}$
for all $\phi\in\mc D(\he n)$. Obviously, the map $T\to\ccheck T$ is
continuous from $\mc D'(\he n)$ to $\mc D'(\he n)$.
}

 Following e.g. \cite{folland_stein}, p. 15, we can define a group
convolution in $\he n$: if, for instance, $f\in\mc D(\he n)$ and
$g\in L^1_{\mathrm{loc}}(\he n)$, we set
\begin{equation}\label{group convolution}
f\ast g(p):=\int f(q)g(q^{-1}\cdot p)\,dq\quad\mbox{for $q\in \he n$}.
\end{equation}
We {recall} that, if, say, $g$ is a smooth function and $P$
is a left invariant differential operator, then
$$
P(f\ast g)= f\ast Pg.
$$
We {also recall} that the convolution {is well} defined
when $f,g\in\mc D'(\he n)$, provided at least one of them
has compact support.

In this case the following identities
hold
\begin{itemize}
\item[i)] \begin{equation}\label{convolutions var}
\Scal{f\ast g}{\phi} = \Scal{g}{\ccheck f\ast\phi}
\quad
\mbox{and}
\quad
\Scal{f\ast g}{\phi} = \Scal{f}{\phi\ast \ccheck g}
\end{equation}
 for any test function $\phi$.
 Analogously, for any  function $\phi \in \mc D$,
 \begin{equation}\label{jan16 eq:6}
 \Scal{f\ast g}{\phi} = \Scal{g}{\ccheck f\ast\phi}
 \qquad \mbox{if $f\in\mc S'(\he n)$ and $g\in\mc S(\he n)$}{\color{blue},}
\end{equation}
(see \cite{schwartz}, p. 248)
$\mc S'\ast\mc S\subset \mc O_M\subset \mc E \cap \mc S' $
and $\mc S\ast\mc D\subset \mc S$,
{where $\mc O_M$ denotes the space of the smooth functions slowly
increasing at infinity (see \cite{treves}, Theorem 25.5, \cite{schwartz} p.243).
Analogously,
 \begin{equation}\label{jan16 eq:6 bis}
 \Scal{f\ast g}{\phi} = \Scal{g}{\ccheck f\ast\phi}
 \qquad \mbox{if $f\in\mc S(\he n)$ and $g\in\mc S'(\he n)$}
\end{equation}
(notice that $\mc S\ast\mc D\subset \mc S\ast\mc S\subset \mc S$).
}
{Indeed}, { by \cite{treves}, Remark 28.3,} 
there exists
a sequence $(g_k)_{k\in \mathbb N}$ { in $\mc D$} such that
$g_k \to g$ in $\mc S'$, so that $f\ast g_k \to f\ast g$
in $\mc D'$ as $k\to \infty$. Since $\ccheck f\ast\phi
\in S$, the assertion follows from \eqref{convolutions var};
{
\item[ii)]  If $\psi\in\mc D(\he n)\subset \mc E'(\he n)$ and $h\in\mc E(\he n)\subset \mc D'(\he n)$,
then $\Scal{\psi}{h} = \Scal{h}{\psi}$, so that, if $\phi,\psi\in \mc D(\he n)$ and $g\in
\mc D'(\he n)$,  \eqref{convolutions var} yields
\begin{equation}\label{convolutions var bis}
 \Scal{\psi\ast \ccheck g}{\phi}= \Scal{\phi}{\psi\ast \ccheck g}=  \Scal{\phi\ast \ccheck g}{\psi}.
\end{equation}

}
 \item[iii)]  if the convolution $g\ast f$ is well defined, then \begin{equation}\label{check conv}
 \ccheck f\ast \ccheck g = \ccheck (g\ast f)
 \end{equation}
  \end{itemize}
  
  {
  The notion of convolution can be extended by duality
to currents.
\begin{definition}\label{regolarizzazione di una corrente}
Let $\phi\in\mc D(\he n)$ and
$T\in \mc E'(\he n, E_0^{h})$ be given, and
 denote by $\ccheck \phi$ the function defined by
${\phantom o}^{\mathrm v} \phi(p) := \phi(p^{-1})$
{(if $S$ is a distribution, then $\ccheck S$ is defined
by duality).}
Then we set
$$
\Scal{\phi\ast T}{\alpha}:= \Scal{T}{\ccheck \phi\ast\alpha}
$$
for any $\alpha\in  \mc D(\he n, E_0^{h})$.
\end{definition}

\begin{definition}\label{alpha ast matrix}
Let $h=1,\dots,2n+1$ be fixed, and let $\xi_1^h,\dots,\xi_{N_h}^h$ be
an orthonormal basis of $E_0^h$. If 
$$
\phi:=\big(\phi_{ij}\big)_{i,j=1,\dots, N_m} 
$$
is a matrix-valued distribution,
and
$\alpha=\sum_j\alpha_j\xi_j\in \mc D(\he n, E_0^h)$,
we set
$$
\alpha \ast \phi=\sum_{i,j} \Big(\alpha_j \ast \phi_{ij}\Big) \xi_i.
$$
Obviously, this notion still makes sense whenever all convolutions $\phi_{ij}\ast\alpha_j$
are well defined.
\end{definition}
}
  
\subsection{Folland-Stein-Sobolev spaces and homogeneous kernels}\label{Folland-Stein-Sobolev}

The following sections deal
with Sobolev spaces (the so-called Folland-Stein-Sobolev spaces: see \cite{folland}, \cite{folland_stein})),
and with the calculus for homogeneous kernels (see \cite{christ_et_al}) 
in the more general setting of Carnot groups. Heisenberg groups will
provide special instance. We refer to \cite{folland,folland_stein} for the standard definitions of Sobolev spaces and their H\"older counterpart
$\Gamma_\beta(\he n)$.  Recall that we adopt the following
multi-index notation for higher-order derivatives: if $I =
(i_1,\dots,i_{n})$ is a multi--index, we define
\begin{equation*}
W^I=W_1^{i_1}\cdots
W_{n}^{i_{n}}.
\end{equation*}

{ \begin{definition}
We denote by $\Delta_\he{}$ the positive sub-Laplacian
$$
\Delta_\he{}:= -\, \sum_{i=1}^{2n} W_i^2.
$$
\end{definition}
}

\begin{definition} Let $1\le p \le\infty$
and $m\in\mathbb N$,  $W^{m,p}_{\mathrm{Euc}}(U)$
denotes the usual Sobolev space. 
\end{definition}

%We recall also the notion of 
%(integer order) Folland-Stein Sobolev space (for a general presentation, see e.g. \cite{folland} and \cite{folland_stein}).
%

\begin{definition}\label{integer spaces} If $U\subset \he n$ is an open set, $1\le p \le\infty$
and $m\in\mathbb N$, then
the space $W^{m,p}(U)$
is the space of all $u\in L^p(U)$ such that, with the notation of \eqref{centrata},
$$
W^Iu\in L^p(U)\quad\mbox{for all multi-indices $I$ with } d(I) \le m,
$$
endowed with the natural norm  
$$\|\,u\|_{W^{k,p}(U)}
:= \sum_{d(I) \le m} \|W^I u\|_{L^p(U)}.
$$

\end{definition}

{  Folland-Stein Sobolev spaces enjoy the following properties akin to those
of the usual Euclidean Sobolev spaces (see \cite{folland}, and, e.g. \cite{FSSC_houston}).}

\begin{theorem} If $U\subset \he n$, $1\le p \le  \infty$, and $k\in\mathbb N$, then
\begin{itemize}
\item[i)] $ W^{k,p}(U)$ is a Banach space.
\end{itemize}
In addition, if $p<\infty$,
\begin{itemize}
\item[ii)] $ W^{k,p}(U)\cap C^\infty (U)$ is dense in $ W^{k,p}(U)$;
\item[iii)] if $U=\he n$, then $\mc D(\he n)$ is dense in $ W^{k,p}(U)$;
\item[iv)]  if $1<p<\infty$, then $W^{k,p}(U)$ is reflexive.
\end{itemize}

\end{theorem}

%\begin{theorem}\label{donne}[see \cite{folland}, Theorem 5.15] If $p>Q$, then 
%$$
%W^{1,p}(\he n) \subset L^\infty (\he n)
%$$
%algebraically and topologically.
%\end{theorem}

%Let $h=1,\dots,2n+1$ be fixed, and let $\xi_1,\dots,\xi_{N_h}$ be
%an orthonormal basis of $E_0^h$. If 
%$$
%\phi:=\big(\phi_{ij}\big)_{i,j=1,\dots, N_m} 
%$$
%is a matrix with entries in $L^1_{\mathrm{loc}}(\he n)$,
%and
%$\alpha=\sum_j\alpha_j\xi_j\in \mc D(\he n, E_0^h)$,
%we set
%$$
%\alpha \ast \phi=\sum_{i,j} \Big(\alpha_j \ast \phi_{ij}\Big) \xi_i.
%$$
%Obviously, this notion still makes sense whenever all convolutions $\phi_{ij}\ast\alpha_j$
%are well defined.

\bigskip

%\subsection{Calculus for homogeneous kernels in Carnot groups}

%The following sections deal
%with Sobolev spaces (the so-called Folland-Stein-Sobolev spaces: see \cite{folland}, \cite{folland_stein})),
%and with the calculus for homogeneous kernels (see \cite{christ_et_al}) 
%in the more general setting of Carnot groups. Heisenberg groups will
%provide special instance. We refer to \cite{folland,folland_stein} for the standard definitions of Sobolev spaces and their H\"older versions.  
%

\begin{definition}\label{folland kernels} Following \cite{folland}, \cite{folland_stein}, a kernel of type $\alpha$ is a 
homogeneous distribution of degree $\alpha-Q$
(with respect to {the} group dilations $\delta_r$),
that is smooth outside of the origin.
\end{definition}

The following estimate has been proved in \cite{BFT1}, Lemma 3.7.
It will turn useful in the sequel.
 \begin{lemma}\label{pointwise}
 Let $g$ be a kernel of type $\mu>0$.
 Then, if $f\in \mc D(\he n)$ and $R$ is {\color{brown}a} homogeneous
 polynomial of degree $\ell\ge 0$ in the horizontal derivatives,
 we have
 $$
R( f\ast g)(p)= O(|p|^{\mu-Q-\ell})\quad\mbox{as }p\to\infty.
 $$
 In addition, let $g$ be
 a smooth function in $\he n\setminus\{e\}$
 {  satisfying} the logarithmic estimate
$$|g(p)|\le C(1+|\ln|p|| ), $$ and suppose
its {first order} horizontal derivatives are kernels of {  type} $Q-1$
with respect to group dilations.
Then, if $f\in \mc D(\he n)$ and $R$ is {\color{brown}a} homogeneous
 polynomial of degree $\ell\ge 0$ in the horizontal derivatives,
 we have
 \begin{eqnarray*}
R( f\ast g)(p)&=&O(|p|^{-\ell})\quad \mbox{as }p\to\infty
\quad\mbox{ if $\ell>0$}; \\
R( f\ast g)(p)&=&O(\ln|p| )\quad \mbox{as }p\to\infty
\quad\mbox{ if $\ell=0$}.
\end{eqnarray*}

 \end{lemma}

We set now
$$
\mc S_0 :=\big\{ u\in\mc S\; : \; \int_{\he{n}} x^\alpha u(x)\, dx =0 \big\}
$$
for all monomials $x^\alpha$.

\begin{definition}\label{Kalpha}
If $\alpha\in\mathbb R$ and $\alpha\notin \mathbb Z^+ :=\mathbb N\cup\{0\}$, then we denote by
$\mathbf K^\alpha$ the set of the distributions in $\he{n}$ that are smooth away from
the origin and homogeneous of degree $\alpha$, whereas, if $\alpha\in \mathbb Z^+$, we say 
that $K\in\mc D'(\he{n})$ belongs to $\mathbf K^\alpha$ if has the form
$$
K= \tilde K+p(x)\ln |x|,
$$
where $\tilde K$ is smooth away from
the origin and homogeneous of degree $\alpha$, and $p$ is a homogeneous polynomial of degree
$\alpha$.
\end{definition}

In particular, kernels of type $\alpha$ according to Definition \ref{folland kernels} belong to
 $\mathbf K^{\alpha-Q}$.
%In particular, if $0<\alpha<Q$, and 
% $h(t,x)$ is the heat kernel associated with the sub-Laplacian $\Delta_\G$,
% then (\cite{folland}, Proposition 3.17) the kernel $R_\alpha
%\in L^1_{\mathrm{loc}}(\G)$ defined by
%$$
%R_\alpha(x):= \dfrac{1}{\Gamma(\alpha/2)}
%\int_0^\infty t^{(\alpha/2)-1}h(x,t)\;dt
%$$
%belongs to $\mathbf K^{\alpha-Q}$.

If $K\in\mathbf K^\alpha$, we denote by $\mc O_0(K)$ the operator defined on $\mc S_0$
 by $\mc O_0(K)u :=u\ast K$.
 \begin{proposition}[\cite{christ_et_al}, Proposition 2.2]\label{s0}
$ \mc O_0(K): \mc S_0\to\mc S_0$.
 \end{proposition}
 
A straightforward computation shows that
\begin{lemma}\label{deriv}
If $K\in \mathbf K^\alpha$, and $X^I$ is a left invariant homogeneous
differential operator, then 
$$
X^I \mc O_0(K)
= { \mc O_0}(X^I K), \qquad \mbox{and $X^I K \in \mathbf K^{\alpha-d(I)}$.}
$$
%Moreover,  the core $K_{R,I}$ of $X^I \mc O(K_R)$ satisfies
%$$
%K_{R,I}\sim X^IK,
%$$
%and
%$$
%X^I \mc O(K_R) = \mc O( (X^IK)_R)\quad  \mathrm{mod}
%\, \mc{OC}^{-\infty}.
%$$
\end{lemma}

 \begin{theorem}[see \cite{knapp_stein}, \cite{koranyi_vagi}]\label{knapp_stein} If $K\in \mathbf K^{-Q}$, then
 $\mc O_0(K):L^2(\he{n})\to L^2(\he{n})$.
\end{theorem}

\begin{remark}\label{density negative} We stress that we
{also have}
$$\mc S_0(\he{n})
\subset \dom (\Delta_{\he{}}^{-\alpha/2})\quad \mbox{with }\alpha >0.
$$
Indeed, take $M\in\mathbb N$, $M>\alpha/2$.
If $u\in\mc S_0(\he{n})$, we can write $u=\Delta_{\he{}}^M v$, where
$$
v:  = \big(\mc O_0(R_2)\circ \mc O_0(R_2)\circ\cdots\circ
  \mc O_0(R_2)\big)u \in \mc S_0(\he{n})
$$
($M$ times). Since $v\in \dom (\Delta_{\he{}}^M) \cap \dom
(\Delta_{\he{}}^{M-\alpha/2})$ by density,
then $u=\Delta_{\he{}}^M v\in \dom (\Delta_{\he{}}^{-\alpha/2})$, and
$\Delta_{\he{}}^{M-\alpha/2}v = \Delta_{\he{}}^{-\alpha/2} \Delta_{\he{}}^M v$,
by \cite{folland}, Proposition 3.15, (iii).
   \end{remark}

  \begin{theorem}[see \cite{glowacki} and \cite{christ_et_al}, Theorem 5.11]\label{rockland}
{Take $K\in \mathbf K^{-Q}$ and let}  the following Rockland condition hold:
 for every  nontrivial irreducible unitary representation $\pi$ of $\he{n}$, 
the operator $\overline{\pi_{K}}$ is injective on
$\mathbf C^\infty(\pi)$, the space of smooth vectors of the representation $\pi$.
Then the operator  $\mc O_0(K):L^2(\he{n})\to L^2(\he{n})$ is left invertible.
\end{theorem}
Obviously, if $\mc O_0(K)$ is formally self-adjoint, i.e. if $K=\ccheck K$, then $\mc O_0(K)$
is also right invertible.

 \begin{proposition}[\cite{christ_et_al}, Proposition 2.3]\label{composizione} If $K_i\in \mathbf K^{\alpha_i}$, $i=1,2$,
 then there exists at least one  $K\in \mathbf K^{\alpha_1+\alpha_2+Q}$
such that $$\mc O_0(K_2)\circ \mc O_0(K_1) = \mc O_0(K).$$
It is possible to provide a standard procedure yielding such a $K$ (see \cite{christ_et_al}, p.42).
Following \cite{christ_et_al}, we write $K=K_2\Ast K_1$.
 \end{proposition}
 
 \begin{definition} Throught this paper, if $\mc L$ is an operator acting on functions,
 then we still denote by $\mc L$ the diagonal operator 
 $\big( \delta_{ij}\mc L\big)_{i,j=1,\dots,M_h}$.
 
 \end{definition}

\begin{lemma}\label{bessel S0} If $m >0$ and $u\in \mc S_0(\he n)$, then $(1-\Delta_{\he n})^{-m/2}u\in \mc S_0(\he n)$.

\end{lemma}

\begin{proof} By \cite{folland}, p.185 (3), $(1-\Delta_{\he n})^{-m/2}u = u\ast J_m$, where
$J_m$ is the Bessel potential defined therein. For our purpose, it is important to stress
that
\begin{itemize}
\item[i)]  $J_m \in L^1(\he n)$;
\item[ii)] $J_m(p) = {O(|p|^{-N})}$ for all $N\in \mathbb N$ (see again
\cite{folland} p.185 (2)). 
\end{itemize}

It is easy to see that $u\ast J_m$ is smooth. To prove that $u\ast J_m\in \mc S(\he n)$ we
can follow basically the same arguments we shall use later to prove that all moments of $u\ast J_m$
vanish. Thus we shall not repeat twice the same computations {\color{brown}(}that, by the way, are elementary
though cumbersome).

Thus we have to show that all moments of $u\ast J_m$ vanish.
In the sequel, we denote by $\tilde J$ a smooth function in $\he n\setminus \{e\}$ satisfying i) and ii) above.

To start with, we can write
\begin{equation*}\begin{split}
\int_{\he n}  &\Big ( \int_{\he n} u(y) J_m (y^{-1}x)\, dy\Big) dx =
\int_{\he n}  u(y) \Big( \int_{\he n}  J_m (y^{-1}x)\, dx\Big)dy
\\& = \mbox{(putting $y^{-1}x=\xi$)} \quad  \int_{\he n}  J_m (\xi) \,d\xi \cdot \int_{\he n} u(y)\, dy =0.
\end{split}\end{equation*}

Denote now by $(x_1,\dots,x_{2n}, x_{2n+1} )$ a generic point in $\he n$.

Take for instance
$$
\int_{\he n} x_j \Big ( \int_{\he n} u(y) \tilde J (y^{-1}x)\, dy\Big) dx\qquad\mbox{with {$ j=1,\dots,2n$}}.
$$
We write
\begin{equation*}\begin{split}
\int_{\he n} x_j & \Big ( \int_{\he n} u(y) \tilde J (y^{-1}x)\, dy\Big)=
\int_{\he n}  u(y) \Big( \int_{\he n} x_j\, \tilde J (y^{-1}x)\, dx\Big)dy
\\&
= \int_{\he n}  u(y) \Big( \int_{\he n} (x_j - y_j)\, \tilde J (y^{-1}x)\, dx\Big)dy
+ \int_{\he n}  y_j\,u(y) \Big( \int_{\he n}  \tilde J (y^{-1}x)\, dx\Big)dy.
\end{split}\end{equation*}
As above,
$$
\int_{\he n}  y_j\,u(y) \Big( \int_{\he n}  \tilde J (y^{-1}x)\, dx\Big)dy = 0,
$$
since $u\in\mc S_0(\he n)$. On the other hand
\begin{equation*}\begin{split}
\int_{\he n}  u(y) & \Big( \int_{\he n} (x_j - y_j)\, \tilde J (y^{-1}x)\, dx\Big)dy 
=
\int_{\he n}  u(y)  \Big( \int_{\he n} (y^{-1}x)_j \, \tilde J (y^{-1}x)\, dx\Big)dy 
\\& =
\quad  \int_{\he n}  \xi_j \tilde J (\xi) \,d\xi \cdot \int_{\he n} u(y)\, dy =0.
 \end{split}\end{equation*}
If $j=2n+1$ the argument is similar, but requires some further tricks.
We write:
\begin{equation*}\begin{split}
x_{2n+1} &= (y^{-1}x)_{2n+1} + y_{2n+1} + \frac12 \sum_{j=1}^n 
\big( y_j x_{n+j} - x_j y_{n+j} \big)
\\&
=  (y^{-1}x)_{2n+1} + y_{2n+1} + \frac12 \sum_{j=1}^n 
\big( y_j (x-y)_{n+j} - (x-y)_j y_{n+j} \big).
\end{split}\end{equation*}
Therefore
\begin{equation*}\begin{split}
\int_{\he n} x_{2n+1} &\Big ( \int_{\he n} u(y) \tilde J (y^{-1}x)\, dy\Big) dx
\\&
= \int_{\he n}  \Big ( \int_{\he n} u(y)  (y^{-1}x)_{2n+1} \tilde J (y^{-1}x)\, dy\Big) dx
\\&
 +\int_{\he n}  \Big ( \int_{\he n} u(y)  y_{2n+1} \tilde J (y^{-1}x)\, dy\Big) dx
 \\&
 + \frac12 \sum_{j=1}^n \Big\{ \int_{\he n}  \Big ( \int_{\he n} u(y)  y_{j} (x-y)_{n+j} \tilde J (y^{-1}x)\, dy\Big) dx
 \\&
 +  \int_{\he n}  \Big ( \int_{\he n} u(y)  y_{n+j} (x-y)_{j} \tilde J (y^{-1}x)\, dy\Big) dx =0
\end{split}\end{equation*}
arguing as above. Thus, by iteration, $(1-\Delta_{\he n})^{-m/2}u\in \mc S_0(\he n)$.
\end{proof}

\begin{lemma}\label{density S0} If $m\ge 0$, then $\mc S_0$ is dense in $W^{m,2}(\he n)$.

\end{lemma}

\begin{proof}
If $m=0$ then the assertion follows straighforwardly via Fourier transform.
Suppose now $m>0$ and let $v\in W^{m,2}(\he n)$ be normal to $\mc S_0$, i.e.
\begin{equation}\label{density eq}
\scal{(1- \Delta_\he{})^{m} v}{u}_{L^2(\he n)}=0 \qquad\mbox{for all $u\in \mc S_0$.}
\end{equation}

Let now $\phi\in\mc S_0$ arbitrary. By Lemma \ref{bessel S0}, 
we can take in \eqref{density eq} $u:=(1- \Delta_\he{})^{-m} \phi\in S_0$. Therefore
$$
\scal{v}{\phi}_{L^2(\he n)} = \scal{(1- \Delta_\he{})^{m} v}{u}_{L^2(\he n)} =0,
$$
and the assertion follows since $\mc S_0$ is dense in $L^2(\he n)$.
\end{proof}

\begin{definition} Once a basis of $E_0^\bullet $ is fixed, and $1\le p \le \infty$,
we denote by $L^p(\he n, E_0^\bullet)$
the space of all sections of $E_0^\bullet$ such that their
components with respect to the given basis belong to
$L^p(\he n)$, endowed with its natural norm. Clearly, this definition
is independent of  the choice of the basis itself. The notations, 
$\mc D(\he n, E_0^\bullet)$, $\mc S(\he n, E_0^\bullet)$,  $\mc S_0(\he n, E_0^\bullet)$,
as well as
$W^{m,p}(\he{n}, E_0^\bullet)$ have the same meaning.
\end{definition}

\bibliographystyle{amsplain}

\bibliography{biblio}

\end{document}